\documentclass[12pt,reqno]{amsart}

\usepackage{epsfig}

\newcommand{\R}{{\mathbb R}}
\newcommand{\C}{{\mathbb C}}

\newcommand{\Z}{{\mathbb Z}}

\newcommand{\dbar}{\bar\partial}
\newcommand{\ddbar}{\partial\dbar}

\renewcommand{\phi}{\varphi}

\newcommand{\ccal}{\mathcal{C}}

\newcommand{\hcal}{\mathcal{H}}

\newcommand{\ncal}{\mathcal{N}}

\newcommand\Ha{\operatorname{Ha}^{(1)}}

\numberwithin{equation}{section}

\newcommand{\nj}{n_{\mbox{crit}}(\lambda_j)}

\newtheorem{theo}{{\sc Theorem}}
\newtheorem{cor}[theo]{{\sc Corollary}}

\newtheorem{lem}[theo]{{\sc Lemma}}
\newtheorem{prop}[theo]{{\sc Proposition}}

\newenvironment{rem}{\medskip\noindent{\it Remark:\/} }{\medskip}

\newenvironment{claim}{\medskip\noindent{\it Claim:\/} }{\medskip}

\title[Counting nodal lines which touch the boundary of an analytic domain]
{Counting  nodal lines which touch the boundary of an analytic
domain}

\author{John A. Toth}
\address{Department of Mathematics and Statistics, McGill University, Montreal, CANADA } \email{jtoth@math.mcgill.ca} \thanks{Research partially supported by NSERC grant \# OGP0170280 and a William Dawson Fellowship}

\author{Steve Zelditch}
\address{Department of Mathematics, Johns Hopkins University, Baltimore,
MD 21218, USA} \email{ szelditch@jhu.edu}

\thanks{Research partially supported by NSF grant \# DMS-0603850.}

\date{Sept 24, 2007}

\begin{document}
\maketitle

\begin{abstract} We consider the zeros on the boundary $\partial \Omega$ of
a   Neumann  eigenfunction $\phi_{\lambda}$ of a real analytic
plane domain $\Omega$. We prove that the number of its boundary
zeros is $O (\lambda)$ where $-\Delta \phi_{\lambda} = \lambda^2
\phi_{\lambda}$. We also prove that the number of boundary
critical points of either a Neumann or Dirichlet eigenfunction is
$O(\lambda)$. It follows that the number of nodal lines of
$\phi_{\lambda}$  (components of the nodal set) which touch the
boundary is of order $\lambda$. This upper bound is of  the same
order of magnitude as the length of the total  nodal line, but is
the square root of the Courant bound on the number of nodal
components in the interior.  More generally, the results are
proved for piecewise analytic domains.

\end{abstract}

\section{Introduction}
This article is concerned with  the high energy asymptotics of
nodal lines of Neumann (resp. Dirichlet) eigenfunctions
$\phi_{\lambda}$ on piecewise real analytic plane domains $\Omega
\subset \R^2$:
\begin{equation} \left\{ \begin{array}{ll} - \Delta \phi_{\lambda}  = \lambda^2
\phi_{\lambda} \;\; \mbox{in}\;\; \Omega, & \;\;
\\ \\\partial_{\nu} \phi_{\lambda}
 = 0 \;\;(\mbox{resp.} \phi_{\lambda_j} = 0) \; \mbox{on} \;\; \partial \Omega,
\end{array} \right. \end{equation}  Here,
$\partial_{\nu}$ is the interior unit normal. We denote by
$\{\phi_{\lambda_j}\}$ an orthonormal basis of eigenfunctions of the
boundary value problem corresponding to eigenvalues $\lambda_0 <
\lambda_1 \leq \lambda_2 \cdots$ enumerated according to
multiplicity. The nodal set
$$\ncal_{\phi_\lambda} = \{x \in \Omega: \phi_\lambda(x) = 0\} $$  is a curve
(possibly with self-intersections at the {\it singular points})
which intersects the boundary in the set $ \ncal_{\phi_{\lambda_j}} \cap
\partial \Omega$ of boundary nodal points. The motivating problem of this article is the
following: how many nodal lines (i.e. components of the nodal set)
touch the boundary? Since the boundary lies in the nodal set for
Dirichlet boundary conditions, we remove it from the nodal set
before counting components. Henceforth, the number of components
of the nodal set in the Dirichlet case means the number of
components of $\ncal_{\phi_\lambda} \backslash \partial \Omega.$

\begin{theo} \label{COR} Let $\Omega$ be a piecewise analytic domain
and  let  $n_{\partial \Omega}(\lambda_j)$ be the number of
components of the nodal set of the $j$th Neumann or Dirichlet
eigenfunction which intersect $\partial \Omega$.  Then
$n_{\partial \Omega}(\lambda_j) = O(\lambda_j).$
\end{theo}

For generic piecewise analytic plane domains, zero is a regular
value of all the eigenfunctions $\phi_\lambda$, i.e. $\nabla
\phi_{\lambda} \not= 0$ on $\ncal_{\phi_\lambda}$  \cite{U}; we
then call the nodal set regular. Each regular nodal set decomposes
into a disjoint union of connected components which are
homeomorphic either to circles contained in the interior
$\Omega^o$ of $\Omega$ or to intervals intersecting the boundary
in two points. We term the former `closed nodal loops' and the
latter `open nodal lines'. Thus, we  are counting open nodal
lines.  Such open nodal lines might `percolate'  in the sense of
\cite{Ze} (i.e. become infinitely long in the  scaling limit
$\Omega \to \lambda \Omega$), or they might form $\lambda^{-1}$
-`small' half-loops at the boundary. Our methods may be useful in
counting each type of component and it seems an interesting
direction for future work.

For the Neumann problem, the boundary nodal points are the same as
the zeros of the boundary values $\phi_\lambda |_{\partial
\Omega}$ of the eigenfunctions. The number of open nodal lines is
thus twice the number of boundary nodal points. Hence we can count
open nodal lines by counting boundary nodal points. In the Neumann
case, our result follows from:

\begin{theo}\label{BNP}  Suppose that $\Omega \subset \R^2$ is a piecewise real analytic  plane domain.
 Then the number $n(\lambda_j) = \# \ncal_{\phi_{\lambda_j}} \cap \partial
 \Omega$ of zeros of the boundary values  $\phi_{\lambda_j} |_{\partial
\Omega}$ of the $j$th Neumann  eigenfunction satisfies
$n(\lambda_j) \leq C
 \lambda_j$, where $C$ is a constant depending only on $\Omega$.
\end{theo}
This is a more precise version of   Theorem \ref{COR} in cases
such as integrable billiard domains (rectangles, discs, ellipses)
where the entire nodal set is connected due to the large grid of
self-intersection points of the nodal set. The analogous result in
the Dirichlet case is stated in Corollary \ref{BNPD}. Counting
boundary nodal points of eigenfunctions has obvious similarities
to measuring the length of the interior nodal line, and our
results show that the order of magnitude is the same. We recall
that  S. T. Yau conjectured that in all dimensions, the
hypersurface volume should satisfy  $c \lambda_j \leq
\hcal^{n-1}(\ncal_{\phi_j}) \leq C \lambda_j$ for some positive
constants $c, C$ depending only on $(M, g)$ \cite{Y1,Y2}. The
lower bound was proved in dimension two for smooth domains by
Br\"uning-Gomes  \cite{BG} and both the upper and lower bounds
were proved in all dimensions for analytic $(M, g)$ by
Donnelly-Fefferman \cite{DF,DF2} (see also \cite{L}). For general
$C^{\infty}$ Riemannian manifolds $(M, g)$ in dimensions $\geq 3$
there are at present only   exponential bounds \cite{HHL,HS}. Our
methods involve analytic continuation to the complex as in
\cite{DF,DF2,L}, and it is not clear how to extend them to
$C^{\infty}$ domains.

In comparison to the order $O(\lambda_j)$ of the number of
boundary nodal points,  the total number of connected components
of $\ncal_{\phi_{\lambda_j}}$ has the upper bound $O(\lambda_j^2)$
by the Courant nodal domain theorem. Only in very rare cases is it
known whether this upper bound  is achieved (in terms of order of
magnitude).  When the upper bound is achieved,  the number of open
nodal lines in dimension $2$ is of one lower order in $\lambda_j$
 than the  number of closed nodal loops. This effect is known from
 numerical experiments of eigenfunctions and random waves \cite{BGS,FGS}. The only rigorous result we
 know is the recent proof in \cite{NS} that the average number of
 nodal components of a random spherical harmonic is of
 order of magnitude $\lambda_j^2$. In special cases, the number of
 connected components can be much smaller than the Courant bound,
 e.g. two or three for arbitrarily high eigenvalues
 \cite{Lew}.

Our   methods also yield estimates on the number  of critical
points of $\phi_{\lambda_j}$ which occur on the boundary. We
denote the boundary critical set by
 $$\ccal_{\phi_{\lambda_j}}= \{q \in \partial
\Omega:  (d \phi_{\lambda_j}) (q) = 0\}. $$ In the case of Neumann
eigenfunctions,  $q \in \ccal_{ \phi_{\lambda_j}} \iff d (u_{\lambda_j} |_{\partial
\Omega}(q)) = 0 $ since the normal derivative automatically equals
zero on the boundary, while in the Dirichlet case $q \in \ccal_{\phi_{\lambda_j}}
\iff
\partial_{\nu} \phi_{\lambda_j}(q)= 0$ since the boundary is a level
set.

A direct  parallel to Yau's conjecture for interior
critical points of generic analytic metrics  would be the B\'ezout
bound $\# \ccal_{\phi_{\lambda_j} } \leq C \lambda^n$ in dimension
$n$, and $\lambda^{n-1}$ for boundary critical points.  However,
this bound is unstable since the critical point sets do not even
have to be discrete when the eigenfunctions have degenerate
critical points, and the true count in the discrete case might
reflect the size of the determinant of the Hessian at the critical
point. Note that there is no non-trivial lower bound on the number
of interior critical points \cite{JN}. Related complications occur for boundary critical points.
For instance,  the  radial eigenfunctions on the disc are constant
on the boundary; thus,  boundary critical point sets need not be isolated.  We therefore
need to add a non-degeneracy condition on the derivative
$\partial_t (\phi_{\lambda} |_{\partial \Omega})$ to ensure that
its zeros are isolated and can be counted by our methods. 

We phrase the condition in terms of the  Pompeiu problem and Schiffer conjecture, which
asserts the disc is the only smooth plane domain possessing a Neumann eigenfunction which
is constant on the boundary. In  \cite{Ber}) it is proved that the disc
      is the only bounded simply connected plane domain possessing an
       infinite sequence of such Neumann eigenfunctions. We  say that the Neumann
problem for a bounded domain has the  {\it asymptotic Schiffer  property} if there exists $C
> 0$ such that, for all Neumann eigenfunctions $\phi_{\lambda}$,
   \begin{equation} \label{NEUND}
\frac{\|\partial_t \phi_{\lambda}\|_{L^{2} (\partial
\Omega)}}{\|\phi_{\lambda}\|_{L^{2}(\partial \Omega)}}
      \geq e^{- C \lambda} \end{equation}
     Here, the $L^2$ norms refer to the restrictions of the
     eigenfunction to  $\partial \Omega$. 
      It seems plausible that Berenstein's result might extend to this asymptotic Schiffer property,
      i.e. that the 
     disc is  the only example where (\ref{NEUND}) fails for an infinite sequence.

  \begin{theo}\label{CRITS}  Let $\Omega \subset \R^2$ be piecewise real analytic.
  Suppose that $\phi_{\lambda_j}|_{\partial \Omega}$ satisfies the asymptotic Schiffer  condition (\ref{NEUND})  in the Neumann case.
   Then the
   number of $n_{\mbox{crit}}(\lambda_j) = \# \ccal_{\phi_{\lambda_j} }$ of
   critical points of a Neumann or Dirichlet eigenfunction $\phi_{\lambda_j}$
which lie  on $\partial \Omega$ satisfies  $\nj \leq C \lambda_j$
for some $C$ depending only on $\Omega$.
\end{theo}
This is apparently the first general result on the asymptotic
number of  critical points of eigenfunctions as $\lambda_j \to
\infty$.  Our results on boundary
critical points give some positive evidence for the B\'ezout upper
bound in the analytic case, and  simple examples such as the disc
show that there does not exist a non-trivial lower bound (see \S
\ref{EXAMPLES}). In general,  counting critical points is
subtler than counting zeros or singular points (i.e. points where
$\phi_{\lambda_j}(x) = d \phi_{\lambda_j}(x) = 0$; see
\cite{HS,HHL}).

   In the case of Dirichlet eigenfunctions, endpoints of open
nodal lines are always boundary critical points, since they must
be singular points of $\phi_{\lambda_j}$.  Hence, an upper bound
for $\nj$ also gives an upper bound for the number of open nodal
lines.

\begin{cor}\label{BNPD}  Suppose that $\Omega \subset \R^2$ is a piecewise real analytic  plane domain.
Let  $n_{\partial \Omega}(\lambda_j)$ be the number of open nodal
lines of the $j$th Dirichlet eigenfunction, i.e. connected
components of $\{\phi_{\lambda_j} = 0\} \subset \Omega^o$ whose
closure  intersects $\partial \Omega$. Then $n_{\partial
\Omega}(\lambda_j) = O(\lambda_j).$
\end{cor}

The question may arise why we are so concerned with   piecewise
analytic domains $\Omega^2 \subset \R^2$. By this,  we mean a
compact domain with piecewise analytic boundary, i.e. $\partial
\Omega$ is a union of a finite number of piecewise analytic curves
which intersect only at their common endpoints (cf. \cite{HZ}).
Our interest in such domains  is due to the fact that many
important types of domains in classical and quantum billiards,
such as the Bunimovich stadium or Sinai billiard, are only
piecewise analytic. Their nodal sets have been the subject of a
number of numerical studies (e.g. \cite{BGS,FGS}). In \cite{TZ} we
consider piecewise analytic Euclidean plane domains with ergodic
billiards, which can never be fully analytic.

The results stated above are  corollaries of one basic result
concerning the {\it complex zeros and critical points} of analytic
continuations of Cauchy data of eigenfunctions. When $\partial
\Omega \in C^{\omega}$,  the eigenfunctions can be holomorphically
continued to an open tube domain  in $\C^2$ projecting over an
open neighborhood $W$ in $\R^2$ of $\Omega$ which is independent
of the eigenvalue.  We denote by $\Omega_{\C} \subset \C^2$ the
points $\zeta = x + i \xi \in \C^2 $ with $x \in \Omega$.  Then
$\phi_{\lambda_j}(x)$ extends to a holomorphic function
$\phi_{\lambda_j}^{\C}(\zeta)$ where $x \in W$ and where $ |\xi|
\leq \epsilon_0$ for some $\epsilon_0 > 0$.  We  mainly use the
complexifications to obtain upper bounds on real zeros, so are not
concerned with  the maximal $\epsilon_0$, i.e. the `radius of the
Grauert tube' around $\partial \Omega$, and  do not include the
radius in our notation for the complexification.

Assuming $\partial \Omega$ real analytic, we   define the
(interior) complex nodal set by
$$\ncal_{\phi_{\lambda_j}}^{\C} = \{\zeta \in  \Omega_{\C}:
\phi_{\lambda_j}^{\C}(\zeta) = 0 \},  $$ and the (interior)
complex critical point set by
$$\ccal_{\phi_{\lambda_j} }^{\C} = \{\zeta \in  \Omega_{\C}: d \phi_{\lambda_j}^{\C} (\zeta) =
0\}.$$ We are mainly interested in the restriction of
$\phi_{\lambda_j}^{\C}$ to the complexification $(\partial
\Omega)_{\C}$ of the boundary, i.e.  the open complex curve in
$\C^2$ obtained by analytically continuing a real analytic
parameterization $Q: S^1 \to \partial \Omega$. The map $Q$ admits
a holomorphic extension to an annulus $A(\epsilon)$ (see
(\ref{ANNULUS})) around the parameterizing circle $S^1$ and its
image $Q_{\C}(A(\epsilon)) \subset \C^2$ is an annulus in the
complexification of the boundary; it is analogous to a   Grauert
tube around the real analytic boundary  in the sense of
\cite{GS1,LS1}. We then define the boundary complex nodal set by
$$\ncal_{\phi_{\lambda_j}}^{\partial \Omega_{\C}} = \{\zeta \in  \partial \Omega_{\C}:
\phi_{\lambda_j}^{\C}(\zeta) = 0 \},  $$ and the (boundary)
complex critical point set by
$$\ccal_{\phi_{\lambda_j} }^{\partial \Omega_{\C}} = \{\zeta \in  \partial \Omega_{\C} : d \phi_{\lambda_j}^{\C} (\zeta) =
0\}.$$

More generally, we may assume $\partial \Omega$ is piecewise real
analytic and holomorphically extend eigenfunctions to the analytic
continuations of the real analytic boundary arcs. The radii of
these  analytic continuations of course shrink to zero at the
corners.

\begin{theo} \label{mainthm} Suppose that $\Omega \subset \R^2$ is a piecewise real analytic  plane
domain, and denote by $(\partial \Omega)_{\C}$ the union of the
complexifications of its real analytic boundary components.

\begin{enumerate}

\item  Let
  $n(\lambda_j, \partial \Omega_{\C} ) = \# \ncal_{\phi_{\lambda_j}}^{\partial \Omega_{\C}}.$
 Then,
$n(\lambda_j, \partial \Omega_{\C})
 = O(\lambda_j). $ The $O$-symbol is uniform in the radius of
 $(\partial \Omega)_{\C}$.

 \item Suppose that the Neumann eigenfunctions satisfy (\ref{NEUND}) and
  let $n_{\mbox{crit}}(\lambda_j, \partial \Omega_{\C}) = \# \ccal_{\phi_{\lambda_j} }^{\partial
 \Omega_{\C}}$.
 Then, $n_{\mbox{crit}}(\lambda_j, \partial \Omega_{\C})
 = O(\lambda_j). $

\end{enumerate}

\end{theo}

The  theorems on real nodal lines and critical points  follow from
the fact that real zeros and critical points are also complex
zeros and critical points, hence
\begin{equation} n(\lambda_j) \leq n(\lambda_j, \partial \Omega_{\C} );  \;\;\;\; \nj \leq
n_{\mbox{crit}}(\lambda_j, \partial \Omega_{\C}). \end{equation}
All of the results are sharp, and are already obtained for certain
sequences of eigenfunctions  on a disc (see \S \ref{EXAMPLES}). If
the condition (\ref{NEUND}) is not satisfied, the boundary value
of $\phi_{\lambda}$ is $o(e^{- C \lambda})$ close to a constant
for all $\lambda$.  It is very likely that this forces the
boundary values to be constant, but it would take us too far
afield in this article to prove it.

Although our main interest is in counting open nodal lines, the
method of proof of Theorem \ref{mainthm} generalizes from
$\partial \Omega$ to a rather large class of  real analytic curves
$C \subset \Omega$, even when $\partial \Omega$ is not real
analytic. Let us call a real analytic curve $C$ a {\it good} curve
if there exists a constant $a > 0$ so that   \begin{equation}
\label{GOOD} \frac{\|\phi_{\lambda}\|_{L^{2} (\partial
\Omega)}}{\|\phi_{\lambda}\|_{L^{2}(C)}}
     \leq e^{a \lambda} .\end{equation}
     Here, the $L^2$ norms refer to the restrictions of the
     eigenfunction to $C$ and to $\partial \Omega$.
The following result deals with the case where $C \subset
\partial \Omega$ is an {\em interior} real-analytic
   curve. The  real curve $C$ may then
be holomorphically continued to a complex curve  $C_{\C} \subset
\C^2$ obtained by analytically continuing a real analytic
parametrization of $C$.

   \begin{theo}\label{CNP}
 Suppose that $\Omega \subset \R^2$ is a $C^{\infty}$ plane
domain, and let $C \subset \Omega$ be a good  interior real
analytic curve in the sense of (\ref{GOOD}). Let
  $n(\lambda_j, C) = \# \ncal_{\phi_{\lambda_j} } \cap C $ be the number of intersection points of
  the nodal set of the $j$-th Neumann (or Dirichlet) eigenfunction  with $C$.  Then $n(\lambda_j, C) = O(
 \lambda_j)$.
 \end{theo}

Although the upper bounds are sharp for some domains, we do not
present necessary or sufficient conditions on a domain that the
bounds on zeros or critical points are achieved on that domain for
some sequence of eigenfunctions. We do not know any domain for
which they are not achieved, but there are few domains where the
bounds can be explicitly tested. The boundary (or rather its unit
ball bundle) is naturally viewed as a kind of quantum `cross
section' of the wave group \cite{HZ}.   The growth rate of the
modulus and zeros of Cauchy data of complexified eigenfunctions
depend on what kind of   `cross section' the boundary provides. In
work in progress  \cite{TZ}, we show that at least for some
piecewise analytic domains with ergodic billiards, the
  the number  of complex zeros of $\phi_{\lambda_j}^{\C}|_{\partial \Omega_{\C}}$
    is $\sim C \lambda_j$. It seems that this asymptotic reflects
    the fact that the boundary is a representative cross section in this
    case.

We note that some of the methods and results of this paper are restricted
to dimension two. In  higher dimensions, zeros of the Cauchy data are not isolated
and we would have to count numbers of components of the boundary nodal set. This
seems inaccessible at present.

The organization of this article is as follow:  In \S \ref{AC}, we
use the layer potential representations of Cauchy data of
eigenfunctions, or equivalently the representation in terms of the
Calderon projector, to analytically continue eigenfunctions.  The
analytic continuation of the layer potential representation has
previously been studied by Vekua \cite{V}, Garabedian \cite{G},
and in the form we need by  Millar \cite{M1,M2,M2}. The analytic
continuation is  somewhat subtle due to the presence of logarithms
in  the layer potentials, and does not appear to be well-known; so
we present complete details (which are sometimes sketchy in the
original articles). In \S \ref{G}, we relate growth of complex
zeros to growth of the log modulus of the complexified
eigenfunctions. In \S \ref{CNP} - \S \ref{VOLTERRA}, we prove the
main results. The complexified layer potential representation is
used to obtain an upper bound on the growth rate of the
complexified eigenfunctions $\phi_{\lambda}^{\C}$ in a fixed
complex tube around the boundary as $\lambda \to \infty$.  The
estimate is simpler for interior curves (\S \ref{CNP}) since on
the boundary the analytic continuation involves a Volterra
operator that must be inverted. Almost the same method gives
analogous results on critical points; for the sake of brevity, the
argument is only sketched in  \S \ref{CP}.

We would like to thank L. Ehrenpreis, C. Epstein, P. Koosis and M. Zworski  for informative
discussions of analytic continuations of eigenfunctions and logarithmic
 integrals.

\section{\label{EXAMPLES} Examples}

We begin with  examples illustrating some of the issues we face.
Eigenfunctions are only computable in (quantum) completely
integrable cases, and at present the only known examples are the
unit disc, ellipses and rectangles. It is a classical conjecture
of Birhoff that ellipses are the only smooth Euclidean plane
domains with integrable billiards, so one does not expect further
explicitly computable examples. In addition, one can construct
approximate eigenfunctions, or quasi-modes for many further
domains \cite{BB}; it is plausible, although it is not proved
here, that our results extend to real analytic quasi-modes.

\subsection{The unit disc $D$}

The standard   orthonormal basis  of real valued Neumann
eigefunctions is given in polar coordinates by $\phi_{m, n}(r,
\theta) = C_{m, n} \sin m \theta J_m(j'_{m, n} r), $ (resp. $
C_{m, n} \cos m \theta J_m(j'_{m, n} r)$) where $j'_{m,n }$ is the
$n$th critical point of the Bessel function $J_m$ and where
$C_{m,n}$ is the normalizing constant. The eigenvalue is
$\lambda_{m, n}^2 = (j'_{m, n})^2$. The parameter $m$ is referred
to as the angular momentum. Dirichlet eigenfunctions have a
similar form with $j'_{m,n}$ replaced by the $n$th zero $j_{m, n}$
of $J_m$. Nodal loops  correspond to zeros of the radial factor
while open nodal lines correspond to zeros of the angular factor.

If we fix $m$ and let $\lambda_{m, n} \to \infty$ we obtain a
sequence of eigenfunctions of bounded angular momentum but high
energy. In the $\sin$ case (e.g.), the open  nodal lines consist
of the union of rays $C_m = \{\theta = \frac{2\pi k}{m}, k = 0,
\dots, \frac{m-1}{m}\}$ through the $m$th roots of unity.  Hence,
for each $m$ there exist sequences of eigenfunctions with $\lambda
\to \infty$ but with $m$ open nodal lines; hence,  there exists no
lower bound on the number of nodal lines touching the boundary in
terms of the energy. This example also shows that there cannot
exist a general unconditional result counting intersections of
nodal lines with interior curves, since $\phi_{m, n}|_{C_m} \equiv
0$ and hence the `number' of nodal points on the interior curve
$C_m$ is infinite. In particular, $C_m$ is not `good' in the sense
of (\ref{GOOD}).

At the opposite extreme are the whispering gallery modes which
concentrate along the boundary. These are eigenfunctions of
maximal angular momentum (with given energy), and $\lambda_m \sim
m$. As discussed in \cite{BB}, they are asymptotically given by
the real and imaginary parts of  $e^{i \lambda_m s}
Ai_p(\rho^{-1/3} \lambda_m^{2/3} y)$. Here, $Ai_p(y):= Ai(- t_p +
y)$ where  $Ai$ is  the Airy function and $\{- t_p\}$ are its
negative zeros. Also, $s$ is arc-length along $\partial D$, $\rho$
is a normalizing constant and   $y = 1 - r$.
 Whispering gallery modes saturate the upper
bound on the number of open nodal lines.

\subsection{An ellipse}

We express an ellipse in the form $x^2 + \frac{y^2}{1 - a^2} = 1,
\;\;\; 0 \leq a < 1, $ with  foci  at $(x, y) = (\pm a, 0)$. We
define elliptical coordinates $(\phi,\rho)$ by $(x, y) = (a \cos
\phi \cosh \rho, a \sin \phi \sinh \rho). $ Here, $0 \leq \rho
\leq \rho_{\max} = \cosh^{-1} a^{-1}, \;\; 0 \leq \phi \leq 2 \pi.
$ The lines $\rho = const$ are confocal ellipses and the lines
$\phi = const$ are confocal hyperbolae. The foci occur at $\phi =
0, \pi$ while the origin occurs at $\rho = 0, \phi =
\frac{\pi}{2}. $

The eigenvalue problem  separates into a pair of Mathieu
equations,
\begin{equation} \label{mathieu}
\left\{ \begin{array}{l} \partial_{\phi}^2 G - c^2 \cos^2 \phi G =
-
\lambda^{2} G \\ \\
\partial_{\rho}^2 F - c^2 \cosh^2 \rho F =
\lambda^{2} F \end{array} \right.
\end{equation}
where $c$ is a certain parameter. The eigenfunctions have the form
$\Psi_{m, n}(\phi, \rho) = C_{m,n} F_{m,n}(\rho) \cdot
G_{m,n}(\phi)$ where, $F_{m,n}(\rho) = \,Ce_m(\rho, \frac{k_n
c}{2})$ and $ G_{m,n}(\phi) = ce_m(\phi, \frac{k_n c}{2})$ (and
their sin analogues). Here,   $ce_m, Ce_m$ are special Mathieu
functions (cf. \cite{C} (3.10)-(3.2)).  The Neumann or Dirichlet
boundary conditions determine the eigenvalue parameteres $k_n c$.
The nodal lines are of course given by $\{G = 0\} \cup \{F = 0\}$.
For more details and computer graphics of elliptic bouncing ball
modes we refer to \cite{C}; the original work was done by
Keller-Rubinow.

A new feature in comparison to the disc  is the existence of
Gaussian beams (a bouncing ball mode) along the minor axis, which
is a stable elliptic bouncing ball orbit. Such bouncing ball modes
do not exist in the disc and occur when $m$ is fixed and $n \to
\infty$.

The eigenfunctions which are the Gaussian beams are characterized
as follows: Since the minor axis is
$$I= \{ (\rho,\phi) \in [0,\rho_{\max}] \times [0,2\pi]; \,  \phi = \frac{\pi}{2} \}.$$
one looks for eigenfunctions with mass concentrated along this
interval. Consider the extremal energy levels that satisfy
$$ c^{2}\lambda^{-2} = 1 + {\mathcal O}(\lambda^{-1}).$$
One rewrites
(\ref{mathieu}) in the form:

\begin{equation} \label{mathieu2}
\left\{ \begin{array}{l} -\lambda^{-2} \partial_{\phi}^2 G  +  (\lambda^{-2}c^2 \cos^2 \phi - 1) G = 0\\ \\
-\lambda^{-2} \partial_{\rho}^2 F +  ( \lambda^{-2}c^2 \cosh^2 \rho + 1)
F =0 \end{array} \right.
\end{equation}
Given the choice of energy level,
$$ \lambda^{-2} c^{2} \cos^{2} \phi - 1 = \cos^{2} \phi -1 + {\mathcal O}(\lambda^{-1}) = -\sin^{2}(\phi) + {\mathcal O}(\lambda^{-1}).$$
The potential $V_1(\phi) = \cos^{2} \phi - 1$ has a nondegenerate
minimum at $\phi = \frac{\pi}{2}$. So, the solutions $G = G_{m,n}$
to the first equation in (\ref{mathieu2}) are asymptotic to ground
state Hermite functions. More precisely,
\begin{equation} \label{asymptotics1}
G_{m,n}(\phi;\lambda) = c_{m,n}(\lambda)e^{- \lambda \cos^{2} \phi }  (
1 + {\mathcal O}(\lambda^{-1})).
\end{equation}
In the second equation in (\ref{mathieu2}) the potential is
$V_{2}(\rho) = \cosh^{2} \rho + 1 + {\mathcal O}(\lambda^{-1}) >0$ for
$\lambda \geq \lambda_{0}$ sufficiently large. In this case the
solution has purely oscillatory asymptotics:
\begin{equation} \label{asymptotics2}
F_{m,n}(\rho;\lambda) = e^{i \lambda \int_{0}^{\rho} \sqrt{
\cosh^{2} x + 1} dx} a_{+}(\rho;\lambda) + e^{-i \lambda
\int_{0}^{\rho} \sqrt{ \cosh^{2} x + 1} dx} a_{-}(\rho;\lambda)
\end{equation}
where $a_{\pm}(\rho;\lambda) \sim \sum_{j=0}^{\infty}
a_{\pm,j}(\rho) \lambda^{-j}$ are determined by the Dirichlet or
Neumann boundary conditions. Moreover, from the
$L^{2}$-normalization condition  $\int_{I}
|\Psi_{m,n}(\rho,\frac{\pi}{2})|^{2} d\rho = 1$  it follows that
$c_{m,n}(\lambda) \sim \lambda^{1/4}.$

 From (\ref{asymptotics1}) and (\ref{asymptotics2}), the Gaussian beams  are roughly  asymptotic to superpositions of $e^{\pm i
k s} e^{- \lambda y^2}$ (cf. \cite{BB}),  where  $s$ denotes
arc-length along the bouncing ball orbit and $y$ denotes the Fermi
normal coordinate. It follows that outside a tube of any given
radius $\epsilon > 0$, the Gaussian beam decays on the order
$O(e^{- \lambda \epsilon^2})$. Hence on any curve $C$ which is
disjoint from the bouncing ball orbit, the restriction of the
Gaussian beam to $C$  saturates the description of a `good'
analytic curve.

\subsection{Remarks}

\noindent{\bf (i)} The goodness requirement (\ref{GOOD}) on any interior curve $C \subset \Omega $ is implied by an exponential growth estimate involving only the Cauchy
data $(\phi_{\lambda}|_C, \partial_{\nu} \phi_{\lambda}|_C)$ along $C$. This is a consequence of the following unique continuation argument.

 Assume that $C$ is a closed curve in the interior of $\Omega$.  Let $U_{C}$ be the domain with boundary $C \cup \partial \Omega$. It follows  from the Sobolev restriction theorem that
 \begin{equation} \label{restriction}
 \| \phi_{\lambda} \|^{2}_{L^{2}(\partial \Omega )}  \leq  C \| \phi_{\lambda} \|_{H^{1/2}(U_{C})}^{2}.
 \end{equation}
 Let $int(C)$ be the interior of the domain bounded by the curve $C$ and take $x \in int(C)$. From the potential layer formula
  (see \ref{GREENSFORMULA}) $\phi_{\lambda}(x) = \int_{C} ( \partial_{\nu(q)} G(x,q;\lambda) \phi_{\lambda}(q) - G(x,q;\lambda) \partial_{\nu_q} \phi_{\lambda}(q) ) d\sigma(q)$ and so, by squaring both sides, using the  bounds $|G(x,q,\lambda)| = {\mathcal O}(1)$ and $|\partial_{\nu_q} G(x,q;\lambda)| ={\mathcal O}( \lambda^{1/2})$ and applying Cauchy Schwartz, one gets
  \begin{equation} \label{cont1}
  \| \phi_{\lambda} \|_{L^{2}(int(C))}^{2} \leq C \lambda ( \|\phi_{\lambda}\|_{L^{2}(C)}^{2} + \| \partial_{\nu} \phi_{\lambda} \|_{L^{2}(C)}^{2}).
  \end{equation}
  By a standard Carleman estimate/unique continuation argument \cite{EZ, Ta, Ta2}:
$$ \| \phi_{\lambda} \|_{H^{1/2}(U_{C})}^{2} \leq e^{C\lambda} \|\phi_\lambda \|_{L^{2}(int(C)) }^{2} \leq  \lambda e^{C\lambda} ( \|\phi_{\lambda}\|_{L^{2}(C)}^{2} + \| \partial_{\nu} \phi_{\lambda} \|_{L^{2}(C)}^{2}),$$
where the last inequality follows from (\ref{cont1}).
Inserting the last bound on the RHS in (\ref{restriction})  yields the comparison estimate relating Cauchy data along $C$ and $\partial \Omega$:
\begin{equation} \label{mainexpbound1}
 \|\phi_{\lambda}\|_{L^{2}(\partial \Omega)}  \leq \, e^{C\lambda}  \, ( \|\phi_{\lambda}\|_{L^{2}(C)} + \| \partial_{\nu} \phi_{\lambda} \|_{L^{2}(C)}).
 \end{equation}
 As an immediate consequence of (\ref{mainexpbound1}) we note that (\ref{GOOD}) follows from the exponential bound
 \begin{equation} \label{mainexpbound2}
 \| \partial_{\nu} \phi_{\lambda}\|_{L^{2}(C)} \leq \, e^{C\lambda}  \| \phi_{\lambda} \|_{L^{2}(C)}
\end{equation}
involving only Cauchy data along $C$.


\vspace{3mm}

A natural question is whether (\ref{mainexpbound2}) is automatically satisfied when $\phi_{\lambda}$
does not vanish identically on $C$? We hope to address this in future
 work.
\medskip

\medskip

\noindent{\bf (ii)} In the case of ellipses, it is elementary to
obtain the analytic continuations of the boundary values of
Neumann eigenfunctions to the complexification of the boundary
ellipse. One sees that  the growth rate of these analytic
continuations are determined  by the angular momenta of the
eigenfunctions, i.e. the eigenvalue of the boundary Laplacian, and
not by the interior eigenvalue. The estimates of this article are
however in terms of the interior eigenvalue. It would be
interesting to strengthen the estimates of this paper to obtain
precise estimates on the growth rate of analytically continuations
of boundary values of eigenfunctions.

\section{\label{AC} Holomorphic extensions of eigenfunctions to Grauert tubes}

It is classical that solutions of the Helmholtz equation on a
Euclidean domain are real analytic in the interior and hence their
restrictions to interior  real analytic  curves admit holomorphic
extensions to the complexification of the curves. A classical
presentation can be found in \cite{G} \S 5.2, and to some extent
we try to conform to its notation; see also
\cite{MN,S,Lew2,Mor,V}. Moreover, the Cauchy data along the
boundary of eigenfunctions satisfying Dirichlet or Neumann
boundary conditions
 admit analytic continuations into a uniform tube in the complexification of the
boundary,   independent of the eigenvalue. We now use integral
representations for  the analytic continuations to  obtain  upper
bounds for the growth rate of the complexified Cauchy data of
 eigenfunctions in a fixed Grauert tube of interior curves or curves on
 the boundary as $\lambda \to \infty$.

 \subsection{Complexification of domains $\Omega$ and analytic
 curves $C \subset \overline{\Omega}$}

 When $(M, g)$ is a real analytic Riemannian manifold without boundary, then there
 exists a complexification $M_{\C}$ of $M$ as an (open) complex
 manifold, and Laplace eigenfunctions extend to a maximal tube in
 $M_{\C}$. This has been studied in  \cite{Bou,GS1,GS2,GLS,Z}
 using the analytic continuation of the geodesic flow and wave
 group.
 The notion of maximal tube depends on the construction of
  a special  plurisubharmonic exhaustion function
 $\rho$ adapted to the metric  $g$ in the sense  that $i \ddbar \rho $ restricts to the totally real
 submanifold $M \subset M_{\C}$ as  $g$. The function
 $\sqrt{\rho}(t)$ is the distance from $t$ to
 $\bar{t}$.

The analogous results for manifolds with boundary have not
apparently been studied before. In this section,  we study the
complexification of the boundary and analytic continuations  of
the Cauchy data of eigenfunctions on the boundary (or on interior
curves) for domains $\Omega
 \subset \R^2$ with the Euclidean metric. The complexification of
 the ambient space is of course $\C^2$ and its Grauert tube function is
 $|\Im t|$. The novel features concern the influence of the
 boundary on analytic extensions of eigenfunctions.

   We  adopt the following
notation from Garabedian \cite{G} and  Millar \cite{M1,M2}: We
denote  points in $\R^2$ and also in $\C^2$ by  $(x, y)$. We
further write
 $z = x + i y, z^* = x - i y$. Note that $z, z^*$ are independent  holomorphic coordinates
 on $\C^2$ and are characteristic coordinates for the Laplacian, in that the
 Laplacian analytically extends to $\frac{\partial^2}{\partial z \partial z^*}$.
  When the boundary is real analytic,
or when we are dealing with a  closed analytic curve $C \subset
\Omega$, we parametrize $C$  by a real analytic parameterization
$Q: {\mathbb S}^{1} \rightarrow {\mathbb R}^2 \simeq \C$. We
complexify $C$ by holomorphically extending the parametrization
 to $Q^{\C}$ on an annulus
\begin{equation} \label{ANNULUS}  A(\epsilon) := \{ t \in {\mathbb C}; e^{-\epsilon} < |t| <
e^{\epsilon} \}, \end{equation}  for $\epsilon >0$ small enough.
Note that $\bar{Q}$ extends holomorphically to $A(\epsilon)$ as
$Q^*$. Throughout the paper, the subscript  $_{\C}$ or superscript
$^{\C}$ denotes the holomorphic continuation of a curve or
function;  sometimes we omit the sub or superscript for notational
simplicity.

 We also use the notation
\begin{equation} \label{equivalent}
 q(t) =  Q(e^{it})
 \end{equation}
 to write the parametrization as a periodic function on $[0, 2
 \pi]$ and often put $q(s) = q_{1}(s) + i q_{2}(s),
  \bar{q}(s) = q_{1}(s) - i q_{2}(s)$. This parametrization
   analytically continues to a periodic function on  $[0,2\pi] +i [-\epsilon, \epsilon]$.
  Both  $s$ and  $t$ can be
either real or complex.

    We denote the boundary data of the eigenfunction by  $$u_{\lambda}(s) = \phi_{\lambda}(q_{1}(s),
q_{2}(s)) \;\;(Neumann); \;\;\; u_{\lambda}(s) =
\partial_{\nu} \phi_{\lambda}(q_{1}(s), q_{2}(s))\;\; (Dirichlet),  $$
and again write $u_{\lambda}^{\C}$ for its holomorphic extension.

Next, we put $r^2 = (\xi -x)^2 + (\eta - y)^2$ so that for $s \in {\mathbb R}$ and $t \in {\mathbb C},$
\begin{equation} \label{rsquared} r^2(s,t) = (q(s) - q(t))
 (\bar{q}(s) - q^*(t)).  \end{equation}
 We denote by $\frac{d}{dn}$ the
 not-necessarily-unit normal derivative in the direction  $i  q'(s)$. Thus,
in terms of the notation $\frac{\partial}{\partial \nu}$ above,
$\frac{d}{dn} = |q'(s)| \frac{\partial}{\partial \nu}. $ One has
$$\frac{d}{ds} \log r = \frac{1}{2} \left[ \frac{q'(s)}{q(s) - q(t)} +
\frac{\overline{q}'(s)}{\overline {q(s)} - q^*(t)} \right],
\;\;\frac{\partial}{\partial n} \log r = \frac{- i }{2} \left[
\frac{q'(s)}{q(s) - q(t)} -
\frac{\overline{q}'(s)}{\overline{q(s)} - q^*(t)} \right]. $$ When
we are using an arc-length parameterization, $\frac{d}{dn} =
\frac{\partial}{\partial \nu}$.

To clarify the notation,  we consider the case of $S^1 = \partial
\Omega$. Then, $q(s) = e^{i s}$,  $t = \theta + i \xi$, $q(\theta
+ i \xi) = e^{i (\theta + i \xi)}, q^* = e^{- i (\theta + i \xi)},
\overline{q^*} = e^{ i (\theta - i \xi)}, $ and
$$\begin{array}{lll}  r^2(s,\theta + i \xi) =  (e^{i (\theta + i \xi)} -
e^{i s}) (e^{- i (\theta + i \xi)} - e^{- i s})  &=& 4 \sin^2
\frac{(\theta - s + i \xi)}{2}. \end{array} $$ Thus, $\log r^2 =
\log (4 \sin^2 \frac{(\theta - s + i \xi)}{2}). $ Clearly,
$\frac{d}{ds} = \frac{d}{d\theta}$, so
$$\frac{d}{ds} \log r^2 =  \left[ \frac{ i e^{is}}{e^{is} - e^{i(\theta + i \xi}} +
\frac{- i e^{-i s}}{e^{-i s} - e^{- i (\theta + i \xi)}} \right],
\;\;\frac{\partial}{\partial \nu} \log r = \frac{- i }{2}  \left[
\frac{ i e^{is}}{e^{is} - e^{i(\theta + i \xi}} - \frac{- i e^{-i
s}}{e^{-i s} - e^{- i (\theta + i \xi)}} \right]. $$

 \subsection{Layer potential representations}

 The representations we use are analytic continuations of layer potential representations
 of solutions. Let $G(\lambda, x, y)$ be any `Green's function' for the Helmholtz equation
 on $\Omega$, i.e. a solution
 of $(- \Delta - \lambda^2) G(\lambda, x, y) = \delta_x(y)$ with $x, y \in \bar{\Omega}$.
We will always use restrictions of the global  Euclidean Green's
functions on $\R^2$.

For any closed curve $C \subset \bar{\Omega}$ bounding a domain
$\mbox{int} (C)$  Green's formula gives
\begin{equation}\label{GREENSFORMULA}  \phi_\lambda (x) = \int_C \left(\partial_{\nu}
G(\lambda, x, q) \phi_\lambda(q) - G(\lambda, x, q)
\partial_{\nu} \phi_\lambda(q) \right) d\sigma(q), \end{equation}
 where $d\sigma$ is
arc-length  and where $\partial_{\nu}$ is the normal derivative by
the interior unit normal. We will analytically continue this
formula for various choices of analytic $C$.

  When $C = \partial \Omega$ and the eigenfunctions satisfy standard
  boundary conditions, the formula simplifies.   In the case of
   Neumann eigenfunctions $\phi_\lambda$ in $\Omega,$ which are
   emphasized here,
\begin{equation} \label{green1}
\phi_\lambda(x)= \int_{\partial \Omega} \frac{\partial}{\partial
\nu_{\tilde{q}} } G(\lambda, x, \tilde{q}) u_\lambda(\tilde{q}) d\sigma (\tilde{q}),\;\; x \in
\Omega^o\;\; (\mbox{Neumann}).
\end{equation}
 Here, and henceforth, we denote the
restriction of $\phi$ to $\partial \Omega$ by $u_j$.
In the Dirichlet case, the corresponding formula is
\begin{equation} \label{green1d}
\phi_\lambda(x)= - \int_{\partial \Omega}  G(\lambda, x,
\tilde{q})\frac{\partial}{\partial \nu_{\tilde{q}}}  \phi_\lambda(\tilde{q}) d\sigma (\tilde{q}),\;\;
(\mbox{Dirichlet})
\end{equation}

 To obtain
concrete representations we need to choose $G$, and we often
choose the real ambient Euclidean Green's function $S$ (in the
notation of \cite{G}, \S 5),
\begin{equation} \label{potential1}
S(\lambda, x, y; \xi, \eta)  = - Y^{(0)}(\lambda r),
\end{equation}
where $r = \sqrt{z z^*}$ is  the distance function and where
\begin{equation} \label{Y0} Y^{(0)} (\lambda r ) = J_0(\lambda r)  \log (kr)  -  \sum_{m =
1}^{\infty} \frac{(-1)^m (1 + 2 \cdots + \frac{1}{m}) (\lambda
r)^{2m}}{4^m (m!)^2} \end{equation} is the Bessel function of
order zero of the second kind.  The Euclidean Green's function has
the form
\begin{equation} \label{GAB} S(\lambda, \xi, \eta;
x, y) = \linebreak A(\lambda, \xi, \eta; x, y) \, \log \frac{1}{r} + B(\lambda, \xi,
\eta; x, y),\end{equation}
with
$$\begin{array}{l} A = J_0(\lambda r) : =  \sum_{k = 0}^{\infty} (-1)^k \frac{(\lambda r)^{2k}}{2^{2k}
(k!)^2}, \\ \\

 \;\; B = -  \sum_{m =
1}^{\infty} \frac{(-1)^m (1 + 2 \cdots + \frac{1}{m}) (\lambda
r)^{2m}}{4^m (m!)^2} + J_0(\lambda r) \log \lambda. \end{array}$$
The coefficient $A$ is known as the Riemann function and $B$ is
the Bessel function of order zero of the first kind.

Sometimes it is more convenient to use the (complex valued)
Euclidean outgoing Green's function  $\Ha_0(k z)$, where $\Ha_0 =
J_0 + i Y_0$ is the Hankel function of order zero. It has the same
form (\ref{GAB}) and only differs by the addition of the even
entire function $J_0$ to the $B$ term.

 By  the `jumps' formulae, the double layer
potential  $\frac{\partial}{\partial \nu_{\tilde{q}}} S(\lambda, x, \tilde{q})$
restricts to the boundary as $\frac{1}{2} \delta_q(\tilde{q})  +
\frac{\partial}{\partial \nu_{\tilde{q}}} S(\lambda, q, \tilde{q})$ (see e.g.
\cite{T}). Hence in the Neumann case the boundary values $u_j$
satisfy,
\begin{equation} \label{green1c} u_\lambda(q)=  2 \int_{\partial \Omega}
\frac{\partial}{\partial \nu_{\tilde{q}}} S(\lambda, q, \tilde{q}) u_\lambda(\tilde{q})
d\sigma(\tilde{q})\;\; (\mbox{Neumann}).
\end{equation}
In the Dirichlet case, one takes the normal derivative of $\phi_j$
at the boundary to get a similar formula for $\partial_{\nu}
\phi_j |_{\partial \Omega}$,  with a sign change on the right
side. We have,
\begin{equation}  \begin{gathered}
\frac{\partial}{\partial \nu_{\tilde{q}}} S(\lambda, q, \tilde{q}) =  - \lambda
Y^{(1)}_1 (\lambda  r) \cos \angle(q -\tilde{q}, \nu_{\tilde{q}})
\end{gathered}\label{Neumann-F}\end{equation}
where
\begin{equation} \label{HANKEL} \begin{array}{l}
\pi Y_1(z) = \frac{-2}{z} + 2 \;J_1(z) \; (\log (z/2) + {\bf \gamma})
 - \sum_{k = 1}^{\infty} (-1)^{k+1} \frac{1}{k! (k - 1)!} \; (z/2)^{2k - 1} [\frac{1}{k} +
2 \sum_{m = 1}^k \frac{1}{m}]. \end{array} \end{equation}  Here, $\gamma$ is Euler's
constant. As is  well-known,  the pole of $Y^1$   is cancelled by the $\cos
\angle(q -\tilde{q}, \nu_{\tilde{q}})$ factor.

If instead we use the Hankel free outgoing Green's function, then
in place of (\ref{Neumann-F}) we get the kernel
\begin{equation} \label{HANKELINT} \begin{array}{lll} N(\lambda, q(s), q(s'))
&=& \partial_{\nu_y} \Ha_{0}(\lambda |q(s) - y|)|_{y = q(s')} \\ &&\\
&=& - \lambda \Ha_{1} (\lambda |q(s) - q(s')|)
\cos \angle(q(s') -q (s),
\nu_{q(s')}).
\end{array} \end{equation}


We now consider the analytic continuations of these formulae.
First we parametrize $C$ by a real analytic parameterization.
Since $J_0$ is even, $ A(\xi, \eta, x, y)$ admits the holomorphic
continuation
  $R(\zeta,
\zeta^*, z, z^*)$. To simplify notation put
$R(q(s), \bar{q}(s), z, z^*) : = R(s; z, z^*)$.

\subsection{Interior curves.} \label{interiorcase} First we consider the simple problem of analytically
continuing the representations (\ref{green1}) and (\ref{green1d}).
We are interested in restrictions to real analytic closed  curves
$C \subset \Omega^o$.

Let $q: [0,2\pi] \rightarrow C$ denote a  real-analytic parametrization of the interior curve $C$. For $\epsilon >0$ sufficiently small, we consider the annulus $A(\epsilon) = \{ t \in \C; e^{-\epsilon} < |t| < e^{\epsilon} \}$ and the corresponding complexification of $C$ given by
$$C_{\mathbb C} = \{ Q^{\C}(e^{i(s + i\tau)}) \in q^{\C}(A(\epsilon)); \,  Q^{\C}(e^{i(s+i\tau)}) = q^{\C}(s + i\tau); \,\, 0 \leq s \leq 2\pi, \, -\epsilon < \tau < \epsilon \},$$
where, we recall that $^{\C}$ denotes holomorphic continuation.
 Since $C$ is assumed to be an interior curve, it follows by
  compactness of $\partial \Omega$ that for $|\Im q^{\C}|$
  sufficiently small, $r^{2}|_{C_{\C} \times \partial \Omega} \neq 0$.
  As a result, one can choose a globally defined holomorphic branch for
   $\log r$ and so, the holomorphic continuation formula for Neumann
    eigenfunctions in this case follows immediately from (\ref{green1}):
\begin{equation} \label{potential2}
\phi_\lambda^{\C}(q^{C}(t))= \int_{\partial \Omega} \frac{\partial}{\partial \nu_{\tilde{q}} }
S(\lambda, q^{\C}(t), \tilde{q}(s)) \, u_\lambda(s) \, d\sigma(s).
\end{equation}

\subsection{The case  $C = \partial \Omega$.} \label{millar1}

We let $q(s)$ denote a real analytic paramaterization of $\partial
\Omega$.  We use the arc-length parameterization so that
$d\sigma(s) = ds$. From (\ref{GAB}), we can write (\ref{green1c})
as

\begin{equation}\label{BASIC}  \begin{array}{lll} u_{\lambda}(t) &=& \frac{1}{2\pi} \int_0^{\ell} (- u_{\lambda}(s) \frac{\partial
  A}{\partial \nu}(s, t)) \log r^2 ds \\ &&  \\
 &&  - \frac{1}{\pi} \int_0^{\ell} u_{\lambda}(s) A(s, t) \frac{1}{r} \frac{\partial
  r}{\partial \nu} ds - \frac{1}{\pi} \int_0^{\ell} (- u_{\lambda}(s) \frac{\partial
  B}{\partial \nu} (s, t) ds. \end{array} \end{equation}
With different choices of $B$ the same formula is valid for the
outgoing Green's function as well.

Since the integral is now over $\partial \Omega$,   the
logarithmic factor in $S$ gives rise to a multi-valued integrand,
and it is not obvious that the representation can be
holomorphically extended. We first concentrate on the case where
$C = \partial \Omega$ when $\partial \Omega$ is real analytic, and
later consider the case  where $C$ is an analytic arc of $\partial
\Omega$ when it is piecewise real analytic.

So assume at first that  $C = \partial \Omega$ is real analytic,
and as above let $u_{\lambda}(s) = \phi_{\lambda}(q(s))$
 be the boundary traces of the Neumann eigenfunctions. Our goal is to analytically
  continue the representation (\ref{BASIC}).

Millar's formula for the holomorphic continuation of the Cauchy
data  is as follows; let
\begin{equation} \label{PHI} \Phi(t; z, z^*) = \int_0^t u_{\lambda}(s) \frac{\partial}{\partial n}
R(s, z, z^*) ds. \end{equation}

\begin{prop} \cite{M1,M2} \label{MILLAR} The boundary data $u = u_{\lambda}$ of the eigenfunctions
of the Neumann problem admit the following holomorphic extension
to a uniform tube around $\partial \Omega$ in its complexification
$(\partial \Omega)_{\C}$: (for $\Im t
> 0, < 0$)
\begin{equation} \label{MILLARFORM} \begin{array}{lll}  u_{\lambda}^{\C}(t)
  &=& \pm i \Phi(t, q(t), q^*(t)) + \int_0^{\ell} \left[ \Phi(s;
q(t), q^*(t)) + i u_{\lambda} (s) R(s, q(t), q^*(t)) \right]
\frac{q'(s)}{q(s) - q(t)} ds \\ && \\
&& +  \int_0^{\ell} \left[ \Phi(s; q(t), \bar{q}(t)) - i u_{\lambda} (s)
R(s, q(t), q^*(t)) \right] \frac{\bar{q}'(s)}{\bar{q}(s) -
q^*(t)}
ds\\ && \\
&& - 2   \int_0^{\ell} u_{\lambda}(s) \frac{\partial B}{\partial n}(s; q(t),
q^* (t))  ds. \end{array} \end{equation}
\end{prop}

\begin{proof} Since we depend crucially on this Proposition, and
since  it does not appear to be well-known, or to be proved in
detail in \cite{M1,M2,V},  we supply the details of the proof.

We will analytically continue the formula (\ref{BASIC}).
 Although $u$ is real analytic on $\partial \Omega$ and hence
admits an analytic continuation to a small complex `tube'
$(\partial \Omega)_{\C}$, it is not clear that the representation
(\ref{BASIC}) can be extended analytically due to singularities of
the integrand. Moreover it is not clear that the right side of
(\ref{MILLARFORM}) is in fact complex analytic. The main task in
the proof is to clarify these points.

We begin by showing that the last two terms of (\ref{BASIC})
analytically continue in a straightforward way.

\begin{lem} \label{easy}  The integrals (i)\; $\frac{1}{\pi} \int_0^{\ell}  u_{\lambda}(s) A(s, t) \frac{1}{r} \frac{\partial
  r}{\partial \nu} (s, t) ds$, resp. (ii)\; $ \frac{1}{\pi} \int_0^{\ell}  u_{\lambda}(s) \frac{\partial
  B}{\partial \nu} (s, t) ds$,  are real analytic on the parameter interval $S^1$
  parametrizing $\partial \Omega$ and are holomorphically extended to an annulus by the formulae
$$ (i)\;\;  \int_0^{\ell}  i u_{\lambda}(s) R(s, q(t), q^{*}(t)) \left(
\frac{q'(s)}{q(s) - q(t)} -  \frac{\bar{q}'(s)}{\bar{q}(s) -
q^*(t)} \right)
ds,$$  resp. $$(ii) \;\; - 2   \int_0^{\ell} u_{\lambda}(s) \frac{\partial B}{\partial \nu}(s;
q(t), q^{*}(t))  ds.$$

  \end{lem}

\begin{proof}

Any derivative of $\log r^2$ is unambiguously defined and we have
$$\frac{1}{r} \frac{\partial r}{\partial n} = \frac{\partial \log r}{\partial n} = \frac{1}{2i}
[\frac{q'(s)}{q(s) - q(t)} - \frac{\bar{q}'(s)}{\bar{q}(s) - q^*(t)}]. $$
In the real domain, $q^*(t) = \bar{q}(t)$, so
$$\frac{1}{r} \frac{\partial r}{\partial n}  = \Im \frac{q'(s)}{q(s) - q(t)}. $$
We recall that $\frac{\partial r}{\partial \nu} = |q'(s)|^{-1} \frac{\partial r}{\partial n}.$
In terms of the real parametrization $q(\phi)$,  $$\frac{\partial r}{\partial \nu} =  \cos
\angle(q(\phi_2) -q (\phi_1), \nu_{q(\phi_2)}) $$ vanishes to order one
on the diagonal in the real domain so that $\frac{1}{r} \frac{\partial r}{\partial \nu}$ is
real and continuous. In complex notation, the same statement follows from the fact that
$$\lim_{t \to s} \frac{q(s) - q(t)}{s - t} =  q'(s) \implies \frac{q'(s)}{q(s) - q(t)} = \frac{1}{s - t}
 + O(1), \;\; (s \to t), $$
 where $\frac{1}{s - t}$ is real when $s, t \in \R$. Hence,
  $\Im \frac{q'(s)}{q(s) - q(t)}$ is continuous for $s, t \in [0, \ell]$ and since $q(s), q(t)$
  are real analytic, the map $s \to [\frac{q'(s)}{q(s) - q(t)} - \frac{\bar{q}'(s)}{\bar{q}(s) - q^*(t)}]$
  is a continuous map from $s \in [0, \ell]$ to the space of holomorphic functions of $t$.

Since $A = J_0(k r)$ is an analytic function of $r^2$,   $A
 \frac{\partial \log r}{\partial \nu}$
has the form $$F(r^2) \;  \frac{1}{2i}
[\frac{q'(s)}{q(s) - q(t)} - \frac{\bar{q}'(s)}{\bar{q}(s) - q^*(t)}], $$
for an analytic function $F$. Clearly, $F(r^2(s, t))$ is also a continuous map from
$s \in [0, \ell]$ to the space of holomorphic functions of $t$. Hence, so is the product
and therefore so is the integral over $s \in [0, \ell]$ of the product.

 Similarly for case (ii). In this case,   $B$ is an entire function $H(r^2)$
of $r^2$ which is of the form $r^2 h(r^2)$ for another entire $h$.  Hence, $$\frac{\partial
B}{\partial \nu} = \;\; r^2 H'(r^2) \times \frac{1}{2i}
[\frac{q'(s)}{q(s) - q(t)} - \frac{\bar{q}'(s)}{\bar{q}(s) - q^*(t)}]. $$ So the integral (ii) also admits
an analytic continuation.

\end{proof}

Thus, the difficulty in analytic continuation of the representation is entirely with the integral
$\int_0^{\ell} (- u(s) \frac{\partial
  A}{\partial \nu}(s, t)) \log r^2 ds$. Due to the logarithm, the analytic continuation of the
  integrand is multi-valued in any neighborhood of $\partial \Omega$.
  Nevertheless,
  the integral admits a single-valued analytic continuation there.

  \begin{lem}  The integral $\int_0^{\ell} (- u(s) \frac{\partial
  A}{\partial \nu}(s, t)) \log r^2 ds$ extends to a holomorphic function of $t$ in
  a neighborhood of $\partial \Omega$ in $(\partial \Omega)_{\C}$ given by
  $$\pm i \Phi(t, q(t), q^*(t)) + \int_0^{\ell} \left[ \Phi(s;
q(t), q^*(t))\right] \frac{q'(s)}{q(s) - q(t)} ds
 +  \int_0^{\ell} \left[ \Phi(s; q(t), q^*(t))  \right] \frac{\bar{q}'(s)}{\bar{q}(s) -
q^*(t)} ds,$$ where, $\pm$ corresponds to  $\Im t >0,<0. $
\end{lem}

\begin{proof}

   We first observe that
\begin{equation} \label{ANU} \frac{\partial A}{\partial \nu} = J_0'(r) \frac{\partial
r}{\partial \nu} = J_1(r) \frac{\partial r}{\partial \nu}.
\end{equation}
Now $J_1$ is odd in $r$ so
(\ref{ANU}) has the form
\begin{equation} \label{DADNU} F(r^2) \;\; r \frac{\partial r}{\partial \nu}
 =F(r^2) \;\; r^2\; \frac{\partial \log r}{\partial
\nu} = F(r^2)  \; r^2\; \left(\frac{1}{2i}
[\frac{q'(s)}{q(s) - q(t)} - \frac{\bar{q}'(s)}{\bar{q}(s) - q^*(t)}]\right),
\end{equation}  where $F$ is a holomorphic function.
From (\ref{DADNU}) it follows that $\frac{\partial A}{\partial \nu}(t, s) $
is a smoothly varying family of holomorphic functions of $t$ in a sufficiently small annulus.

Thus, our problem is a special case of the general problem of
analytically continuing the integral $\int_0^{\ell} f(s) \log
r^2(s, t) ds$ where $f$ is real analytic and where $r^2(s, t)$ is
given by (\ref{rsquared}). In our case, $f(s)$ also depends
holomorphically on $t$ but this does not affect the analytic
continuation issue.

To define the analytic continuation, we specify a branch  $L(s,
t)$  of the multi-valued analytic continuation of $\log r^2(s,t)$.
We first  slit the complex parameter annulus through the vertical
segment through  $0$  to obtain the complex $t$ parameter strip $I
= [0, \ell] + i (- \epsilon, \epsilon)$. Our integrals only
involve  pairs $(s, t) \in [0, \ell] \times I$. We then remove the
the set $0 \leq s < \Re t$. For fixed $s$, these cuts  disconnect
the $t$-annulus into two strips. On one strip, $0 \leq \Re t \leq
s$ in the parameter interval, while $s \leq \Re t \leq \ell$ in
the other. We then further slit these  strips along the real
segment $\Im t = 0$ of $I$, to obtain four strips or `quadrants'.
In  the right `half-plane' where  $s
> \Re t$, we  define $\Im L (s,t) $ so that it is continuous in
the right `half-plane'
  and  tends $0$ as $\Im t \to 0$
from either top or bottom.   In the slit left `half-plane',  $\{s
< \Re t\} \backslash [0, \Re t]$ we define  $L(s, t)$ by
continuation from the right half plane. It then tends to
 $\mp 2\pi$ as $\Im t \to 0$ from  above, resp. below the cut along
 the real axis $s < \Re t$.

  To illustrate, we consider the basic case of the circle,
 where $q(t) = e^{i t}$ and where we are defining
 $\arg \left( (e^{is} - e^{i t})(e^{-is } - e^{- it}) \right)$.  We fix $\Re t =
 t_0$ and consider the map $(s, \tau) \to (e^{is} - e^{i t})(e^{-is } - e^{-
 it})$
 where $t = t_0 + i \tau$. In the `first quadrant' $s > t_0, \Im t
 > 0$, this map is anti-holmorophic and takes the real axis $\Im t
 = 0$ to the positive real axis and the `imaginary axis',  $s =
 t_0$ and
 $\Im t > 0$, to the negative real axis. Since the map is anti-holomorphic,
 the image of a counter-clockwise path in the first quadrant from the real to imaginary
 axis is a clock-wise path from the positive real axis to the
 negative real axis, so the arg equals $- \pi$ on the imaginary
 axis. As the path in the domain moves counter-clockwise in the second quadrant to $s <
 t_0, \Im t = 0$ the image path moves to argument $- 2\pi$.
 Similarly, the
continuation  in the fourth and third quadrants leads to a value
 of $2 \pi$ on the axis $s < t_0$.

 The following claim is the key one to obtain a single valued analytic
 continuation  (cf.  Millar \cite{M2}).

\begin{claim}\label{CLAIM} If $f$ admits an analytic continuation to a
neighborhood of $\partial \Omega$, then  the integral
$\int_0^{\ell} f(s) \log r^2(s, t) ds$ admits an analytic
continuation to a neighborhood of $\partial \Omega$ in $(\partial
\Omega)_{\C}$ by
\begin{equation} \label{mainintegral}
\int_0^{\ell} f(s) \log r^2(s, t) ds \to  \int_0^{\ell} f(s) L (s, t) ds \pm 2 \pi i \int_0^t f(s) ds,
\end{equation}
where the path from $0$ to $t$ is defined in the integral is the
same as the path used to analytically continue $\log r^2(s, t)$,
and where the $+$ sign is taken for $\Im t
> 0$ and the $-$ sign when $\Im t < 0$.
\end{claim}
First, we check the periodicity of the right side in (\ref{mainintegral}). This follows from the fact that
$$\int_0^{\ell} f(s) L(s,\ell + i\Im t) ds  \pm  2\pi i \int_{0}^{\ell+ i \Im t} f(s) ds = \int_{0}^{\ell}  L(s, i \Im t) f(s) ds$$
$$ \pm 2\pi i \int_{0}^{\ell +i\Im t} f(s) ds \mp 2\pi i \int_{0}^{\ell} f(s) ds $$
 $$ = \int_{0}^{\ell}  L(s, i \Im t) f(s) ds \pm 2\pi i
\int_{0}^{i \Im t} f(s)ds,  $$

 where, the last   line  follows from the Cauchy formula and the
 fact that $f(t + \ell) = f(t).$

To prove the claim it suffices to show that the right side in (\ref{mainintegral}) is

\begin{itemize}

\item (i)  Holomorphic in the upper annulus $\Im t > 0$;

\item (ii) Holomorphic in the lower annulus $\Im t < 0$;

\item (iii) Continuous in the whole annulus, and  restricts to
$\int_0^{\ell} f(s) \log r^2(s, t) ds$ for real $t$.

\end{itemize}

 Let us first
recall why this is sufficient:
  \begin{lem} \label{Muk} (see e.g. \cite{Mu}) Let $\Omega \subset
 \C$ be a domain, and suppose
 that $C \subset \Omega$ is a closed curve separating $\Omega$ into two domains $D^+,
D^- $ with common boundary $C$. Suppose that $F^+, F^-$  are
holomorphic on $D^+$, resp. $D^-$ and that $F^+ = F^-$ on $C$.
Then the function $F$ defined by $F = F^+ |_{D^+}, F^- |_{D^-}$
and $F^{\pm}$ on $C$  is holomoprhic on all of $\Omega$. In other
words, a sectionally holomorphic function which is continuous is
holomorphic.
 \end{lem}

Let us prove (iii) first, since it explains the second term on the
right side of Claim \ref{CLAIM}.  With no loss of generality,
suppose that $\Im t \to 0^+$ with $t \to t_0$. Then
$$ \int_0^{\ell} f(s) L (s, t) ds + 2 \pi i \int_0^t f(s) ds
\to
 \int_0^{\ell} f(s) L  (s, t_0) ds + 2 \pi i \int_0^{t_0} f(s)
 ds, $$ and we must show that
 $$\int_0^{\ell} f(s) L (s, t_0) ds + 2 \pi i \int_0^{t_0} f(s)
 ds = \int_0^{\ell} f(s) \log r^2(s, t_0) ds. $$
 Here, $\arg r^2(t,s) = 0$ while $\Im L (s, t)$ equals zero for $s
 \geq t$ and equals $-2 \pi $ for $s \leq t$. Hence, the imaginary part
 of the left side cancels and we obtain the right side.

Now let us prove (i)-(ii). Since the proofs are essentially the
same we only prove (i).

We first note that all  branches of analytic continuation of $\log
r^2(s,t)$ differ by constants in $2\pi i \Z$. Hence, if
 the period $\langle f \rangle :=
\frac{1}{\ell}\int_0^{\ell} f(s) ds$ of $f$ vanishes, then all
choices of branch of $\log r^2$ give the same
 value of the integral $\int_0^{\ell} f(s) \log r^2(s, t) ds$.
Similarly, the integral  $\int_0^t f(s) ds$  is only multi-valued
due to the period  of $f$. Hence, when the $\langle f \rangle =
0$, both terms on the right side of the claim are well-defined
independently of any choice of integration path or branch of $\log
r^2$. Since we can write $f = (f - \langle f \rangle) + \langle f
\rangle $, we only  need to show:
\begin{enumerate}

\item $ \int_0^{\ell} f(s) L (s, t) ds \pm  2 \pi i \int_0^{t}
f(s)
 ds $ is  holomorphic for $\Im t > 0$ when
 $\langle f \rangle = 0$;

\item $ \int_0^{\ell}  L (s, t) ds \pm 2 \pi i \int_0^t ds$ is
single-valued holmorphic function for $\Im t > 0$.

\end{enumerate}

To prove (1), we assume $\langle f \rangle = 0$ and  let $F(t) =
\int_0^t f(s) ds$ be the (well-defined) primitive of $F$ in the
annulus. We then integrate by parts in the first integral to
obtain
\begin{equation} \label{DERIV}\begin{array}{l}  \int_0^{\ell} F'(s) L (s, t) ds \\ \\ =   F(\ell)
L (\ell, t) - F(0) L (0, t) - \int_0^{\ell} F(s) \frac{q'(s)}{q(s)
- q(t)} ds - \int_0^{\ell} F(s)
\frac{\overline{q}'(s)}{\overline{q}(s) - q^*(t)} ds  \\  \\=
 - \int_0^{\ell} F(s)
\frac{q'(s)}{q(s) - q(t)} ds - \int_0^{\ell} F(s)
\frac{\overline{q}'(s)}{\overline{q}(s) - q^*(t)} ds
\end{array}
\end{equation}
Here, we use that $F(\ell ) = F(0)$ since it is a single-valued
holomorphic function on the annulus and that $F(0) = 0$ by
definition.
 It is clear that the
right side of (\ref{DERIV})  is holomorphic in  $\Im t
> 0 $ since  poles occur only when $\Im t = 0.$  Since $F(t)$ is single valued and
holomorphic, this  proves (1).

To prove (2) we write
\begin{equation} \label{ii} \begin{array}{lll} \int_0^{\ell}  L (s, t) ds \pm 2 \pi i \int_0^t ds
&  = &
  \int_0^{2 \pi} \log \frac{Q(e^{is})
- Q(e^{i t})}{e^{is} - e^{i t}} \frac{Q^*(e^{is}) -
Q^*(e^{it})}{e^{-is } - e^{- it}} ds \\ && \\ && +  \int_0^{2 \pi}
\log \left( (e^{is} - e^{i t})(e^{-is } - e^{- it}) \right) ds \pm
2 \pi i \int_0^t ds, \end{array} \end{equation} We observe that
the first term is holomorphic for   $\Im t
> 0$ since the $\arg$ of both numerator and denominator are continued so
that each $\arg$  tends to $  - 2 \pi$ as $\Im t \to 0$ for $s \in
[0, \Re t]$  and so that each $\arg$ tends to zero for $s \in [\Re
t, \ell]$. Hence, the $\arg$ of the ratio tends to zero as $\Im t
\to 0$ in both integrals. Since the $arg$ is only ambiguous up to
a constant in $2 \pi i \Z$, it follows that the $\arg$ of the
ratio is well defined and single valued and the integrand is
well-defined as a single-valued holomorphic function for $\Im t >
0$. Therefore, to prove (2) it suffices to show that for $\Im t
\geq 0$
\begin{equation} \label{main}
   g(t) := \int_0^{2 \pi}
L(s, t) ds + 2 \pi i \int_0^t ds = 0,
\end{equation}
where $L(s, t) = \log \left( (e^{is} - e^{i t})(e^{-is } - e^{-
it}) \right)$ with our choice above of the branch cut at $s = \Re
t$. Note that $g$ is an analytic continuation of the integral
$\int_0^{2 \pi} \log |e^{is} - e^{i  t}|^2 ds = 0$ for real $t$,
so analyticity of $g$ is equivalent to $g = 0$.

This reduces the analysis to the integral $$ \int_0^{2 \pi} \log
\left( (e^{is} - e^{i t})(e^{-is } - e^{- it}) \right) ds =
 \int_0^{2 \pi} \log
\left( 2 - 2 \cos (s - t) \right) ds, $$ where as above the
logarithm is defined by breaking  up the integral into $
\int_0^{\Re t} +
 \int_{\Re t}^{\ell}$ and defining the $\arg$ by the method above.
 Note that formally the integral is constant in $t$ by a change of
 variables but that such a change of variables is not consistent
 with the definition of the logarithm. However,  the
 integrand is a function of $s - t$ and so, $\frac{d}{dt} \log   \left( 2 - 2 \cos (s - t) \right) =  \frac{d}{d(t-s)} \log \left( 2 - 2 \cos (t-s) \right)$ is well-defined independent of the branch of $\log$ used (since these differ by integer
  multiples of $2\pi i$).  Hence,
 $$\begin{array}{lll} \frac{d}{dt}  \int_0^{2 \pi} \log
\left( 2 - 2 \cos (s - t) \right) ds & = &
   -  \int_0^{2 \pi} \frac{d}{ds }\log \left( 2 - 2 \cos (s - t)
\right) ds \\ && \\
& = &  -  \log \left( 2 - 2 \cos (s - t) \right) |_0^{2\pi} \\
&& \\
& = & - 2 \pi i,
\end{array}
$$
by definition of the logarithm. It follows from (\ref{main})  that
$$ \frac{d}{dt} g(t) =  \frac{d}{dt} \int_0^{2 \pi} \log
\left(2 - 2 \cos (s - t) \right) ds + \frac{d}{dt} ( 2\pi i t)  =
- 2\pi i + 2\pi i = 0.$$ Hence $g$ is constant and as noted above
it equals $0$ for real $t$.

\end{proof}

 This completes the proof of the Claim and hence of the
 Proposition.

\end{proof}
\begin{rem} By integrating by parts directly in the integral $Lf(t) := \int_{0}^{\ell} \log r^{2}(s,t) f(s) ds$ for $t$ {\em real} and using that $\int_{0}^{2\pi} \log |e^{is} - e^{it}|^{2} ds =0,$ one gets the formula
\begin{equation} \label{real formula}
Lf(t) =  - \int_{0}^{\ell} \left( \frac{ q'(s)}{q(s)-q(t)} + \frac{\bar{q}'(s)}{ \bar{q}(s) - \bar{q}(t)}  \right)   \cdot  ( F(s) - F(t) ) \, ds + \langle f \rangle \int_{0}^{2\pi} \log \frac{|Q(e^{is})- Q(e^{it})|^{2}}{ |e^{is} - e^{it}|^{2}} ds,
\end{equation}
where, $F(t):= \int_{0}^{t}(f-\langle f \rangle) ds.$  It follows from (\ref{real formula}) that for $t \in [0,\ell ] + i[-\epsilon, \epsilon]$ there is an alternative formula for the analytic continuation of $Lf$ which is given by
\begin{eqnarray} \label{p5}
(Lf)^{\C}(t) = - \int_{0}^{\ell } \left( \frac{ q'(s)}{q(s)-q(t)} + \frac{ \bar{q}'(s)}{ \bar{q}(s) - q^{*}(t)}  \right)   \cdot  ( F(s) - F(t) ) \, ds \\ \nonumber
 + \langle f \rangle \int_{0}^{2\pi} \log \frac{ [Q(e^{is})- Q(e^{it})] [Q^{*}(e^{is})- Q^{*}(e^{it})] }{ [e^{is} - e^{it}][ e^{-is}- e^{-it}]} ds.
\end{eqnarray}
We note that the $\log$ in the second term on the RHS of (\ref{p5}) is defined unambiguously (independent of branch) since as $\Im t \rightarrow 0$ and $s \rightarrow \Re t$ from either side, we have that $\arg  [Q(e^{is})- Q(e^{it})] - \arg [e^{is} - e^{it}] \rightarrow 0$ and so the arguments cancel. The same thing is true for the ratio involving $Q^{*}$.
\end{rem}

\begin{rem} It should be noted that the proof  only
makes use of the fact that $u_{\lambda}$ admits a single-valued
analytic continuation to $(\partial \Omega)_{\C}$.  The right side
of (\ref{BASIC}) defines an operator $N u$ on $C(\partial
\Omega)$, and the proof shows that although $N_{\C}(\lambda,
\zeta, q)$ is multi-valued, its integral against $u$ admits a
single-valued holomorphic extension as long as $u$ admits one. Let
us check the argument in the case of the unit disc, where $N$ is a
convolution operator with kernel
$$N (\lambda, \theta, \phi)   = \sum_{n \in \Z} \lambda \Ha_n
(\lambda) J_n'(\lambda)  e^{i n (\theta - \phi)}.$$ The Fourier
coefficients $\lambda \Ha_n (\lambda) J_n'(\lambda) $ have
only a slow decrease reflecting the singularity of $\log |\sin^2
(\theta - \phi)|$, so the kernel does not admit a holomorphic
extension $e^{i n \theta} \to \zeta^n$ in a pointwise sense.
However,  its integral $N u$  against a real analytic function $u
\sim \sum_n a_n e^{i n \phi}$ admits the  holomorphic extension
$\sum_{n \in \Z} \lambda \Ha_n (\lambda) J_n'(\lambda) a_n
\zeta^n$ to the maximal domain to which $u$ itself holomorphically
extends.

\end{rem}

\subsection{  \bf The case $C\cap \partial \Omega \neq \emptyset$ with
$C \neq \partial \Omega$} \label{millar2}

In this case, we cannot use the boundary conditions to simplify
the integral
    \begin{equation} \label{generaljumpeqn}
    \phi_{\lambda}(q) = 2 \left( \int_{C} \frac{\partial}{\partial \nu_{\tilde{q}}} S (q, \tilde{q};\lambda) \phi_{\lambda}(q) d \sigma(\tilde{q})
    - \int_{C} S(q,\tilde{q};\lambda) \frac{\partial}{\partial \nu_{\tilde{q}}} \phi_{\lambda}(\tilde{q}) \, d\sigma(\tilde{q})\right).
    \end{equation}
The first term is handled exactly as in the case $C = \partial
\Omega$, while the second term (the single layer potential term)
is new.

 To continue the second
      term we write it in the form
    $$ - \frac{1}{2} \int_{C} A(t,s;\lambda)  \, \log r^{2}(s,t) \,
     \partial_{\nu} \phi_{\lambda}(s) d\sigma(s)  - \int_{C} B(t,s;\lambda) \,
     \partial_{\nu}\phi_{\lambda}(s) \,  d\sigma(s).$$
    Since $B = F(r^{2})$ where $F$ is entire, the second term above
    has a straightforward analytic continuation. The first term is
    another case of
     the integrals discussed in the previous section, and its holomorphic continuation is:
    $$- \int_{0}^{t} R(q(t),q^{*}(t);s) L(s,t)
    \partial_{\nu} \phi_{\lambda}(s) \, d\sigma(s) \mp 2\pi i \int_{0}^{t}
     R(q(t),q^{*}(t);s) \partial_{\nu}(s) \, d\sigma(s).$$

\subsection{The case where $C$ is piecewise real analytic.}

We now consider  the case where $C$ is piecewise real-analytic.
This has previously been discussed in \cite{M3}.

By a piecewise analytic embedded closed curve $C$ of length $L(C)$
we mean a curve of the form  $C = \bigcup_{j = 1}^m C_j$ where
\begin{itemize}

\item  $C_j \subset \R^2$ are the maximal real analytic components
of $C$, enumerated in counterclockwise order so that $C_j$
intersects only $C_{j - 1}$ and $C_{j + 1}$.

\item The $C_j$ are  parameterized by $m$ real analytic functions
$q_j(t_j):  [0, \ell_j] \to C_j$ on $m$ parametrizing intervals
(where $\ell_j = L(C_j)$ is the length of $C_j$. We assume $C_j
\cap C_{j - 1} = \{q_j(0) = q_{j - 1}(\ell_j)\}$ when $m \geq 2$.

\end{itemize}

We denote the Cauchy data of the eigenfunction $\phi_{\lambda}$ on
the boundary component $C_j$ by $u^j_{\lambda}$. Our aim is to
analytically continue $u^j_{\lambda}$ to $\bigcup_{j = 1}^m C_{j,
\C}$ where $C_{j, \C}$ is a complexification of the interior of
$C_j$. Thus, as we define it, $C_{\C}$ is pinched at the corner
points and the analytic continuation of the boundary data of
$\phi_{\lambda}$ is somewhat simpler than in the fully analytic
case in that we are analytically continuing to a smaller kind of
neighborhood of $C$.

Millar's formula for the analytic extension of $u^j_{\lambda}$ to
$C_j^{\C}$ in the Neumann case is given by:

\begin{prop} \cite{M3} \label{MILLAR3} The boundary data $u^j_{\lambda}$ of the eigenfunctions
of the Neumann problem admit the following holomorphic extension
to a uniform tube around the interior $C_j^o$ of $C_j$  in its
complexification: (for $\Im t_j
> 0, < 0$)
\begin{equation} \label{MILLARFORM} \begin{array}{lll}  u_{\lambda}^{j, \C}(t_j)
  &=& \pm i \Phi(t_j, q_j(t), q_j^*(t)) + \frac{1}{\pi} \sum_{n = 1}^m \int_0^{\ell_j} [ \Phi(s_n;
q(t_j), q^*(t_j)) \\ && \\ && + i u_{\lambda}^n (s_n) R(s_n,
q(t_j), q^*(t_j)) ]
\frac{q'(s_n)}{q(s_n) - q(t_j)} ds_n \\ && \\
&& +  \int_0^{\ell} \left[ \Phi(s_n; q(t_j), \bar{q}(t_j)) - i
u_{\lambda} (s_n) R(s_n, q(t_j), q^*(t_j)) \right]
\frac{\bar{q}'(s_n)}{\bar{q}(s_n) - q^*(t_j)}
ds\\ && \\
&& - 2   \int_0^{\ell} u_{\lambda}(s) \frac{\partial B}{\partial
n}(s_n; q(t_j), q^* (t_j))  ds_n. \end{array} \end{equation}
\end{prop}

\begin{proof} We only sketch the proof, because it only involves
a small modification of Proposition \ref{MILLAR}.

First,  Green's formulae (\ref{GREENSFORMULA})-(\ref{green1})
remain correct in the piecewise analytic case, with the definition
that on  the $j$th component, the normal derivative is calculated
by taking the limit from within the $j$th component.

The verification of the Millar formula is then similar to the
fully analytic case. The main difference is that we now have pairs
$(s_n, t_j)$ of parameter points which may come from different
intervals. When $n = j$ there is no difference in the argument
except that  $C_{j \C}$ is not an annulus but rather two regions
meeting along a common interval. But the same choice of branch of
the logarithm extends $u_j$ holomorphically above and below the
interval, and the first term on the right side ensures that the
two holomorphic extensions agree on the common interval.
 When $n \not= j$, one defines  $\arg r^2(s_n, t_j) = 0$ for all  real $s_n,  t_j$.
 Since $q_n(s_n) \not= q_j(t_j)$ when $n \not= j$ the logarithm
 extends to a holomorphic function in $t_j$ with this choice of
 branch.

\end{proof}

\vspace{2mm}

\section{\label{G}Growth of zeros and growth of $u_{\lambda}^{\C}$}

 Let $C_{\C}$ be the
complexification of a real analytic curve $C \subset \Omega$. The purpose of this section is to give an upper bound for the number of zeros of $u_{\lambda}^{\C}$ in the annulur region $q^{\C}(A (\epsilon))$ where
$A(\epsilon) = \{ z \in \C; e^{-\epsilon} < |z| < e^{\epsilon} \}$. For  $\lambda_j \in Sp(\sqrt{\Delta})$ and for a region $D \subset
C_{\C}$ we denote by \begin{equation}\label{LITTLEN}  n(\lambda_j,
D) = \#\{q^{\C}(t) \in D : u_{\lambda_j}^{\C}(q^{\C}(t) ) = 0 \}.
\end{equation}

The following estimate is suggested by Lemma 6.1 of
Donnelly-Fefferman \cite{DF}.

\begin{prop} \label{DFnew} Suppose that $C$ is a good real analytic curve
in the sense of (\ref{GOOD}).  Normalize $u_{\lambda}$ so that
$||u_{\lambda}||_{L^2(C)} = 1$. Then, there exists a constant $C(\epsilon) >0$ such that
for any $\epsilon >0$,
 $$n(\lambda, Q^{\C}( A(\epsilon/2) ) ) \leq C(\epsilon)  \max_{ q^{\C}(t) \in   Q^{\C}( A( \epsilon) ) } \log
|u_{\lambda}^{\C}(q^{\C}(t))|. $$ \end{prop}

\begin{proof} Let $G_{\epsilon}$ denote  the Dirichlet Green's function of the `annulus'
$Q^{\C}(A(\epsilon)) $. Also, let $\{a_j\}_{j = 1}^{n(\lambda, Q^{\C}(A(\epsilon/2)) )}$
denote the zeros of $u_{\lambda}^{\C}$ in the sub-annulus
$Q^{\C}(A(\epsilon/2))$. Let $U_{\lambda} =
\frac{u_{\lambda}}{||u_{\lambda}||_{Q^{\C} (A(\epsilon))}}$ where
$||u||_{Q^{\C} (A(\epsilon))} = \max_{\zeta \in Q^{\C} (A(\epsilon))} |u(\zeta)|. $ Then,
$$\begin{array}{lll} \log |U_{\lambda}(q^{\C}(t))| & = &  \int_{ Q^{\C} ( ( A(\epsilon/2) ) )}
G_{\epsilon}(q^{\C}(t), w) \ddbar \log |u_{\lambda}(w)| +
H_{\lambda}(q^{\C}(t)) \\ && \\
& = & \sum_{a_j \in Q^{\C}( A(\epsilon/2) ): u_{\lambda}(a_j) = 0} G_{\epsilon}
(q^{\C}(t), a_j) + H_{\lambda}(q^{\C}(t)), \end{array}$$  since  $\ddbar \log |u_{\lambda}(w)|  = \sum_{a_j \in C_{\C}: u_{\lambda}^{\C}(a_j) = 0} \delta_{a_{j}}$. Moreover, the function
$H_{\lambda}$ is  sub-harmonic  on $Q^{\C} (A(\epsilon))$ since
$$\ddbar H_{\lambda} = \ddbar \log |U_{\lambda}(q^{\C}(t))|  - \sum_{a_j \in Q^{\C}(A(\epsilon/2)): u_{\lambda}(a_j) = 0}  \ddbar G_{\epsilon}
(q^{\C}(t), a_j)  = \sum_{a_j \in Q^{\C} (A(\epsilon)) \backslash
 Q^{\C}(A(\epsilon/2)) } \delta_{a_j} > 0. $$
So, by the maximum principle for subharmonic functions,
$$\max_{Q^{\C} (A(\epsilon))} H_{\lambda} (q^{\C}(t)) \leq \max_{\partial Q^{\C} (A(\epsilon))} H_{\lambda} (q^{\C}(t))
= \max_{\partial Q^{\C} (A(\epsilon))} \log |U_{\lambda}(q^{\C}(t))| = 0. $$   It
follows that
\begin{equation} \log |U_{\lambda}(q^{\C}(t))| \leq \sum_{a_j \in Q^{\C}(A(\epsilon/2) ): u_{\lambda}(a_j) = 0} G_{\epsilon}
(q^{\C}(t), a_j),
\end{equation}
hence that
\begin{equation} \max_{q^{\C}(t) \in Q^{\C}( A(\epsilon/2) )} \log
|U_{\lambda}(q^{\C}(t))| \leq \left( \max_{z, w \in Q^{\C}( A(\epsilon/2) )}
G_{\epsilon}(z,w) \right) \;\; n(\lambda, Q^{\C}( A(\epsilon/2) )).
\end{equation}
Now $G_{\epsilon}(z,w) \leq \max_{w \in Q^{\C}( \partial A(\epsilon) )}
G_{\epsilon}(z,w) = 0$ and $G_{\epsilon}(z,w) < 0$ for $z,w \in
Q^{\C}( A(\epsilon/2) )$. It follows that there exists a constant $\nu(\epsilon) <
0$ so that $ \max_{z, w \in Q^{\C}( A(\epsilon/2) )} G_{\epsilon}(z,w) \leq
\nu(\epsilon). $ Hence,
\begin{equation} \max_{q^{\C}(t) \in Q^{\C}( A(\epsilon/2) )} \log
|U_{\lambda}(Q^{\C}(t))| \leq \nu(\epsilon)  \;\; n(\lambda, Q^{\C}( A(\epsilon/2) )).
\end{equation}
Since both sides are negative, we obtain
\begin{equation}\label{mainbound} \begin{array}{lll}  n(\lambda,
Q^{\C}( A(\epsilon/2) )) & \leq &\frac{1}{|\nu(\epsilon)|}  \left| \max_{q^{\C}(t) \in
Q^{\C}( A(\epsilon/2) )} \log |U_{\lambda}(q^{\C}(t))| \right| \\ && \\
& \leq &\frac{1}{|\nu(\tau)|} \left( \max_{q^{\C}(t) \in Q^{\C} (A(\epsilon))}
\log |u_{\lambda}(q^{\C}(t))| - \max_{q^{\C}(t) \in Q^{\C}( A(\epsilon/2) )} \log
|u_{\lambda}(q^{\C}(t))| \right)\\ && \\
& \leq &\frac{1}{|\nu(\epsilon)|} \;\;  \max_{q^{\C}(t) \in Q^{\C} (A(\epsilon))} \log
|u_{\lambda}(q^{\C}(t))|,
\end{array} \end{equation}
where in the last step we use that $\max_{q^{\C}(t) \in Q^{\C}( A(\epsilon/2) )}
\log |u_{\lambda}(q^{\C}(t))| \geq 0$, which  holds since
$|u_{\lambda}^{\C}| \geq 1$ at some point in $Q^{\C}( A(\epsilon/2) )$. Indeed,   by our
normalization, $\|u_{\lambda}\|_{L^{2}(C)} =1$, and so there must already exist points on the real curve $C$ with $|u_{\lambda}| \geq 1$. Putting $C(\epsilon) = \frac{1}{|\nu (\epsilon)|}$ finishes the proof.
\end{proof}
\vspace{1mm}

\begin{rem} An alternative approach is to use
Jensen's formula,
\begin{equation} \label{nform2e}  \begin{array}{lll}
\int_{0}^{\epsilon} n(\lambda, q_{\C}(A(\rho)) d\rho
  && = M_{f_{\lambda}}(\epsilon)  +
  M_{f_{\lambda}}(-\epsilon) - 2 M_{f_{\lambda}}(0), \end{array}
  \end{equation}
  where
$$M_f(\rho) = \left(\frac{1}{2\pi}
\int_{|z| = e^{\rho}  } \log |f_{\lambda}| d \theta \right). $$
However, this would require an analysis of the real logarithmic
integral $ M_{f_{\lambda}}(0)$. As the example of the Gaussian
beam shows,  $M_{f_{\lambda}}(0)$ may be of order $\lambda$ due to
exponential decay away from the `classically allowed region'. We
plan to analyze such integrals elsewhere.

\end{rem}

\section{\label{CNP} Proof of Theorem \ref{CNP} }
\begin{proof}
We prove Theorem \ref{CNP} before Theorem \ref{BNP} since it is easier.
   The proof of Theorem \ref{CNP} uses the analytically continued potential  layer formula (\ref{potential2}) to bound $\max_{Q^{\C}(A(\epsilon))}  \log |\phi_{\lambda}^{\C}|$ from above. Then, an application of
Proposition \ref{DFnew} gives the result.

\subsection{Upper bounds for the analytically continued eigenfunctions.}  Let $q:[0,2\pi] \rightarrow C$ be an arc-length parametrization.  We  denote the parametrization of the boundary, $\partial \Omega,$ by
$\tilde{q}:[0,2\pi] \rightarrow \partial \Omega$.  In this case,  the formula for holomorphic continuation of eigenfunctions is given by (\ref{potential2}):
\begin{equation} \label{int1}
\phi_\lambda^{\C}(q^{\C}(t))= \int_{\partial \Omega}  N(\lambda,q^{\C}(t),\tilde{q}(s)) \, u_\lambda(s) d\sigma(s),
\end{equation}

From the basic Hankel function formula  (\ref{HANKELINT}) for $N(\lambda,q,\tilde{q})$
 and the standard
integral formula
\begin{equation} \label{integral1}
\Ha_{1}(z) = \left( \frac{2}{\pi z} \right)^{\frac{1}{2}} \,  \frac{e^{i ( z - 3\pi/4)}}{\Gamma (3/2)} \, \int_{0}^{\infty} e^{-s} s^{-\frac{1}{2}} \, (1 - \frac{s}{2iz} )^{\frac{1}{2}} \, ds,
\end{equation}
 one easily gets an asymptotic expansion in $\lambda$ of the form:
\begin{equation} \label{potential3}
N( q^{C}(t),\tilde{q}(s);\lambda)  = e^{i \lambda r( q^{\C}(t), \tilde{q}(s)) }  \, \sum_{j=0}^{k} a_{j}( q^{C}(t), \tilde{q}(s)) \, \lambda^{\frac{1}{2}-j} + {\mathcal O}(e^{i \lambda r( q^{\C}(t), \tilde{q}(s)) } \, \lambda^{\frac{1}{2}-k-1}).
\end{equation}
Note that the expansion in (\ref{potential3}) is valid since
for interior curves,
$$C_{0} := \min_{ (q(t), \tilde{q}(s)) \in C \times \partial \Omega } |q(t) - \tilde{q}(s)|^{2}  >0.$$
 Then,   $\Re r^{2}(q^{C}(t), \tilde{q}(s)) >0$ as long as
 \begin{equation} \label{holbranch}
 | \Im q^{C}(t)|^{2} < C_{0}.
 \end{equation}
  So, the principal square root of $r^{2}$ has a well-defined holomorphic extension to the tube (\ref{holbranch}) containing $C$.  We have denoted this square root by $r$ in (\ref{potential3}).

 Substituting (\ref{potential3})
in the analytically continued single layer potential integral formula (\ref{int1}) proves that
for $t \in A(\epsilon)$ and $\lambda >0$ sufficiently large,
\begin{equation} \label{int2}
\phi^{\C}_{\lambda}(q^{\C}(t)) = (2\pi \lambda^{\frac{1}{2}}) \int_{\partial\Omega} e^{i\lambda r (q^{\C}(t), \tilde{q}(s)) } a_{0}(q^{\C}(t),\tilde{q}(s)) ( 1 + {\mathcal O}(\lambda^{-1})  \, ) \, u_{\lambda}(s)  d\sigma(s).
\end{equation}
Taking absolute values of the integral on the RHS in (\ref{int2}) and applying Cauchy-Schwartz proves

\begin{lem} \label{mainlemma1}
For $t \in [0,l] + i[-\epsilon,\epsilon]$ and $\lambda >0$ sufficiently large
$$|\phi^{\C}_{\lambda}(q^{\C}(t))| \leq C_{1} \lambda^{1/2} \exp \, \lambda \left( \max_{\tilde{q}(s) \in \partial \Omega} \Re  \,  i r(q^{\C}(t), \tilde{q}(s))   \right)  \cdot  \| u_{\lambda} \|_{L^{2}(\partial \Omega)}.$$
\end{lem}


From the pointwise upper bounds in Lemma \ref{mainlemma1}, it is immediate that
\begin{equation} \label{DFupper}
\log \, \max_{q^{\C}(t) \in Q^{\C}(A(\epsilon))} |\phi_{\lambda}^{\C}(q^{\C}(t) )| \leq C_{\max} \lambda  + C_{2} \log \lambda +  \log \| u_{\lambda} \|_{L^{2}(\partial \Omega)},
\end{equation}
where,
$$C_{\max} =  \max_{(q^{\C}(t), \tilde{q}(s)) \in Q^{\C}(A(\epsilon))  \times \partial \Omega} \Re  \,  i r(q^{\C}(t), \tilde{q}(s)).$$


Finally, we use that  $\log
\|u_{\lambda}\|_{L^{2}(\partial \Omega)} = O( \lambda) $ by the assumption that  $C$ is   a good curve and apply Proposition
\ref{DFnew} to get that $n(\lambda,C) = {\mathcal
O}(\lambda).$

\end{proof}

\section{\label{VOLTERRA}  Proof of Theorem \ref{mainthm}: Zeros}

The proof of Theorem \ref{mainthm} is more complicated than that
for interior curves because we need to invert the Volterra
operator  of  Proposition \ref{MILLAR}. We recall that the
analytic continuation of $u$ is the solution of a  Volterra
equation, \begin{equation} \label{VE} (I + K_{\lambda})
u_{\lambda}(t) = U_{\lambda}(t), \end{equation}  where
$U_{\lambda}(t)$ has an explicit analytic continuation, where
\begin{equation} \label{K} K_{\lambda} u_{\lambda}(t) = \int_0^t \frac{\partial R}{\partial \nu}(\lambda, s,
q(t), q^* (t)) u_{\lambda}(s) ds \end{equation}
 in Millar's notation.
Here, $R = A$ is the Riemann function and so explcitly,
$$K_{\lambda} u_{\lambda}(t) = \int_0^t \frac{\partial J_0(\lambda r)}{\partial \nu}(\lambda, s, q(t)
q^* (t)) u_{\lambda}(s) ds.$$ Therefore,
$$K_{\lambda}(t, s) = {\bf 1}_{[0,t]}(s) J_1(\lambda r) r \frac{\partial \log r}{\partial
\nu} (t, s) = {\bf 1}_{[0,t]}(s)  r  J_1(\lambda r)  \left( \frac{q(s)}{q(s) - q(t) } -
\frac{\bar{q}'(s)}{ \bar{q}(s) - q(t)^* }
  \right),  $$ where ${\bf 1}_{[0,t]}$ is the indicator function of the interval $[0,t]$. We note that the pole of $\frac{q(s)}{q(\Re t + is) - q(t)^*}$ at
the upper limit of integration $s = t$ is cancelled because $r
J_1(r)$ begins with $r^2$. So the integrand is regular and
holomorphic along the path of integration.

On the right side of the Volterra equation,
 \begin{equation} \label{RHS} \begin{array}{l} u_{\lambda}^{\C}(t) \mp i \Phi (t,q(t),q^{*}(t))  =  \int_0^{\ell} \left[ \Phi(s;
q(t), \bar{q}(t)) + i u_{\lambda} (s) R(s, q(t), q^*(t)) \right]
\frac{q'(s)}{q(s) - q(t)} ds \\  \\
 +  \int_0^{\ell} \left[ \Phi(s; q(t), q^*(t)) - i
u_{\lambda} (s) R(s, q(t), q^* (t)) \right]
\frac{\bar{q}'(s)}{\bar{q}(s) - q^*(t)}
ds\\  \\
 - 2   \int_0^{\ell} u_{\lambda}(s) \frac{\partial B}{\partial
\nu}(s; q(t), q^* (t))  ds \end{array} \end{equation} the Cauchy
data
 $u_{\lambda}$ is only integrated over   the real domain where by a
 standard Sobolev estimate it has polynomial growth in $\lambda$.
 And further, the Riemann function and other special functions
 occurring there have exponential growth at most given by the
 ambient complexified distance function. The main problem is thus
 to invert the Volterra operator $I + K_{\lambda}$ and to obtain a
 similar growth estimate for $(I + K_{\lambda})^{-1} (RHS)$.

 We first simplify the operator, $K_{\lambda}$.
Given $t = \Re t + i \Im t$ we may choose the contour to go along the
real interval $[0,  \Re t]$ and then to go along the vertical
line segment $\Re t + i s$ for $s \in [0, \Im t]$. This decomposes
$K_{\lambda} = K^{(1)}_{\lambda} + K^{(2)}_{\lambda}$, where

\begin{equation} \label{KONE} K_{\lambda}^{(1)} u_{\lambda}(t) = \int_0^{\Re t} u_{\lambda} (s)
 \frac{\partial}{\partial \nu} R(\lambda; s; q(t), q^*(t)) ds
 \end{equation}
 and where
 $$K_{\lambda}^{(2)} u_{\lambda}(t) = \int_0^{\Im t} u_{\lambda} (\Re t +  is)
 \frac{\partial}{\partial \nu} R(\lambda; \Re t +  is; q(t), q^*(t)) ds. $$
We move the $K^{(1)}_{\lambda}$ term again to the right side since it only
involves the Cauchy data on the real domain.

We now write $t = \Re t + i \Im t$ and treat $\Re t$ as a parameter. We
need to study the mapping properties of $K^{(2)}_{\lambda}$ and
$(I + K^{(2)}_{\lambda})^{-1}$ on the weighted Hilbert space
$L^2([- \epsilon, \epsilon ], e^{- \lambda |\Im t|} d \Im t). $

\subsubsection{Model example.} As a model example, we consider the operator $K_{\lambda} u(y) =
\int_0^y e^{\lambda (y - s) } u(s) ds. $
  A simple calculation
shows that for $n \geq 0,$
$$K_{\lambda}^{n+1} ( y, s) = e^{\lambda(y - s)} \frac{ (y-
s)^n}{n!},
$$
and
$$(I - K_{\lambda})^{-1}(y, s) = e^{(\lambda + 1) (y-s)}. $$
Hence, in the model example, the exponential growth of the kernel
$(I - K_{\lambda})^{-1}(s, \Im t)$ is the same as for $K_{\lambda}(s,\Im t)$.

\subsection{Upper bounds.}
In view of the growth estimate for complex zeros in Proposition \ref{DFnew}, one needs to determine asymptotic pointwise upper bounds for the $|u_{\lambda}^{\C}(t)|$ as $\lambda \rightarrow \infty$.  In this section, we prove:
\begin{lem} \label{mainbound2}
Given $ t \in [0,l] + i[-\epsilon, \epsilon]$ and $\lambda >0$ sufficiently large, there exists a constant $C>0$ such that
$$|u_{\lambda}^{\C}(t)| \leq \exp C \lambda |\Im t|  \cdot \|u_{\lambda}\|_{L^{2}(\partial  \Omega)} .$$
\end{lem}

\begin{proof}
Let $C_{0} >0$ be a constant.  To bound the kernel $K_{\lambda}^{(2)}(\Im t,s)$ we split the analysis into two cases: (i) $|r(\Re t + is,t)| \leq \frac{1}{C_{0}}$ and (ii)  $|r(\Re t + is,t)| \geq \frac{1}{C_0}$.

\subsubsection{The range $|r| \geq \frac{1}{C_0}$}
In this range,  $J_1$ has an asymptotic expansion  given by
$$J_1(\lambda r ) = \sum_{j=0}^{k} \lambda^{-\frac{1}{2} - j}  a_{j}( r ) e^{\lambda i r }+ {\mathcal O}( \lambda^{-\frac{1}{2} - k -1} e^{\lambda  \Im r } ). $$
From the identity
$$\begin{array}{lll} K_{\lambda}^{(2)} &=&
\partial_{\nu} J_0( \lambda r) \\ &&
  =  {\bf 1}_{[0,\Im t]}(s) \, J_1(\lambda r ) \, r \,
\frac{\partial \log r }{\partial \nu} \\ &&
= {\bf 1}_{[0,\Im t]} (s)\, r \, J_1(\lambda r ) \,
[\frac{q'(s)}{q(s) - q(t)} - \frac{\bar{q}'(s)}{\bar{q}(s) - q(t)^*}],\end{array}
$$

it follows that there exists a symbol
$S_{\lambda}$ of order $-\frac{1}{2}$ such that \begin{equation}
\label{KEST} |K_{\lambda}^{(2)}(\Im t, s)| \leq S_{\lambda}(\Im t, s) \, {\bf  1}_{[0,\Im t]}(s) \,
e^{\lambda |\Im r(\Re t + is, t)| }.
\end{equation}
The estimate (\ref{KEST}) is locally uniform in $\Re t$ and the dependence on the parameter $\Re t$ is implicit.
\subsubsection{The range $|r| \leq \frac{1}{C_0}$}

In this range, the asymptotic expansion breaks down when $|r| \ll \frac{1}{\lambda}$  and so, we use the standard integral representation for $J_{1}$ to get the necessary bounds for $K_{\lambda}^{(2)}$. Since

\begin{equation} \label{besselintegral}
J_{1}( \lambda r ) = - \pi i \int_{0}^{\pi} e^{i \lambda r  \, \cos \theta} \, \cos \theta \, d\theta
\end{equation}
and  $|\cos \theta| \leq 1,$ it follows immediately from (\ref{besselintegral}) that in this range,
\begin{equation} \label{besselintegral2}
|J_{1}( \lambda r ) | \leq C e^{\lambda |\Im  r|},
\end{equation}
and so, when $|r(\Re t + is, t)| \leq \frac{1}{C_0},$
\begin{equation} \label{KEST2}
|K_{\lambda}^{(2)} (\Im t,s) | \leq C  {\bf 1}_{[0,\Im t]}(s) \, e^{\lambda |\Im  r(\Re t + is,t)| }.
\end{equation}

Combining the estimates (\ref{KEST}) and (\ref{KEST2}), it follows that
\begin{equation} \label{KESTmain}
|K_{\lambda}^{(2)} (\Im t,s) | \leq C  {\bf 1}_{[0, \Im t]}(s) \, e^{\lambda |\Im  r(\Re t + is,t) |},
\end{equation}
locally uniformly in $|s| + |\Im t|$ and in $\Re t$. Again, the dependence of $K_{\lambda}$ on the parameter $\Re t$ has been suppressed.

\subsubsection{Pointwise estimates for $r$}
By definition,
\begin{equation} \label{taylor1}
r (\Re t + is, t) = \langle q(\Re t + i \Im t) - q(\Re t + i s), q(\Re t + i \Im t) - q(\Re   t + i s) \rangle^{\frac{1}{2}}
\end{equation}
 Taylor expansion around $s=\Im t$  in (\ref{taylor1}) gives
 \begin{equation} \label{modelbound}
| r(\Re t + is, t) | \leq C |\Im t- s|,
\end{equation}
since,
$$q(t) - q(\Re t + i s) = \int_0^1 \frac{d}{dt} q(\Re t + i ( t \Im t + (1- t) s)) dt = (\Im t - s) \int_0^1 q'(\Re t + i
( t \Im t + (1- t) s)) dt. $$
From (\ref{modelbound}) and the bound (\ref{KESTmain})  it follows  that there are constants $C_{j}>0; j=1,2,$ such that
\begin{equation} \label{K2estimate}
|K_{\lambda}^{(2)}(\Im t,s)| \leq C_{1} {\bf 1}_{[0,\Im t]}(s) e^{C_{2}|\Im t - s|}.
\end{equation}

Next, we expand    $(I -
 K^{(2)}_{\lambda})^{-1}$ in a norm convergent geometric series,
$$\sum_{n = 0}^{\infty} [K_{\lambda}^{(2)}]^{n} (\Im t,s), \;\;; [K^{(2)}_{\lambda}]^{n} (\Im t, s): = \int_0^{\Im t} \int_0^{s_n} \cdots \int_0^{s_1}
K^{(2)}_{\lambda}( \Im t, s_n) \cdots K^{(2)}_{\lambda}( s_1, s) \, ds_1 \cdots ds_n. $$
We recall that the $n$-simplex $\Delta_n$  is defined by
$$\{(s_1, \dots, s_n): 0 \leq s_1 \leq s_2 \leq \cdots \leq s_n
\leq 1\}. $$
Let $\Im t  \cdot \Delta_n$ be the scaled simplex. Applying  the estimate (\ref{K2estimate}) to each factor in the above formula for $[K^{(2)}_{\lambda}]^{n}$  gives the following  pointwise bound:
$$|[K^{(2)}_{\lambda}]^{n}(\Im t,s)| \leq \int_{\Im t \cdot \Delta_n} e^{C\lambda (\Im t -s_n)} \cdot e^{C \lambda (s_{n}- s_{n-1}) } \cdots e^{C \lambda (s_{2}-s_{1})} \cdot e^{C \lambda(s_{1}-s)} \,
 ds_1 \cdots ds_n. $$
So, by the model example,
$$| (I - K^{(2)}_{\lambda})^{-1}(\Im t, s)| \leq e^{ (C+1) \lambda  | \Im t - s|} \cdot   {\bf 1}_{[0,\Im t]}(s).$$
To complete the proof of proof of Lemma \ref{mainbound2}, we note that
 for $q^{\C}(t) \in Q^{\C}( A(\epsilon) ),$ the right side of the analytic continuation formula (\ref{RHS})
together with the $K_{\lambda}^{(1)}$ term satisfies the estimate \begin{equation} \label{RHS2} (**) \leq C_{1}  \exp \left( \lambda \max_{q(s) \in \partial \Omega}  \Re \, i r(t,s)  \right) \cdot  \| u_{\lambda} \|_{L^{2}(\partial \Omega)} \leq  C_{1}  e^{C_{2} \lambda |\Im t|},  \end{equation}
since  by our normalization, $\| u_{\lambda} \|_{L^{2}(\partial \Omega)} =1.$
 It follows that \begin{equation} (I -
K_{\lambda}^{(2)})^{-1} (**) \leq C \int_0^{\Im t} e^{ (C+1) \lambda (\Im t - s)}  e^{C_{2} \lambda s} ds \leq C \exp  (
\lambda \max\{C+1, C_{2}\}  |\Im t| ). \end{equation}
This finishes the proof of Lemma \ref{mainbound2}.
\end{proof}

So,  from Lemma \ref{mainbound2},  $\log \max_{q^{\C}(t) \in Q^{\C}(A(\epsilon))} |\phi_{\lambda}^{\C}(t)|  \leq C  \lambda$ and Proposition \ref{DFnew} then implies that
$n(\lambda;\partial \Omega) = {\mathcal O}(\lambda) $. \qed

\section{\label{CP} Critical points: Proof of Theorems \ref{CRITS} and \ref{mainthm} (2) } 

We now prove part (2) of  Theorem
\ref{mainthm} concerning the  growth of critical points.  It immediately implies Theorem \ref{CRITS}.  The
argument is similar to that for counting zeros, the only change
being that we now take the derivative and restrict to the boundary
in (\ref{green1d}). For the sake of brevity we only sketch the
proof.

In the Dirichlet case, the jumps formula for the double layer
potential  gives (\ref{BASIC}) except that now $u_{\lambda_j}$
denotes the restriction to the boundary of $\frac{\partial
\phi_{\lambda_j} }{\partial \nu}.$ We refer to \cite{T,HZ} for  background.
We then define $n(\lambda_j, D)$ as in (\ref{LITTLEN}) but for the
new $u_{\lambda_j}$. The layer potential representation implies
the analogue of Lemma \ref{mainbound2}
 and by Proposition \ref{DFnew} we conclude that
the number of complex zeros (hence real zeros) is
$O(\lambda)$.

In the Neumann  case,  we must  take the tangential derivative
$u_{\lambda}'(t)$. Since the normal derivative is zero, the
critical points of the tangential derivative are critical points
of the eigenfunction along the boundary. The tangential derivative
now has the representation,
\begin{equation}\label{BASICa}  \begin{array}{lll} u_{\lambda}'(t) && = \frac{1}{2\pi} \int_0^{\ell} (- u_{\lambda}(s)
\frac{\partial^2
  A}{\partial t \partial \nu}(s, t)) \log r^2 ds \\ &&  \\
  && + \int_0^{\ell} (- u_{\lambda}(s) \frac{\partial
  A}{\partial \nu}(s, t)) \frac{\partial}{\partial t} \log r^2 ds\\ && \\
 &&  - \frac{1}{\pi} \int_0^{\ell} u_{\lambda}(s) (\frac{\partial}{\partial t} (A(s, t) \frac{1}{r} \frac{\partial
  r}{\partial \nu})) ds - \frac{1}{\pi} \int_0^{\ell} (- u_{\lambda}(s)
  \frac{\partial^2
  B}{\partial \nu \partial t} (s, t) ds. \end{array} \end{equation}
  The analytic continuation of Proposition \ref{MILLAR} applies equally  to the
  equation (\ref{BASICa}) and the analytic continuation of $u'(t)$ has
  the form (\ref{VE}) with a slight change in $K_{\lambda}$. However, the phase
  function is the same, so  the proof of
  Lemma \ref{mainbound2}  applies with only small modifications  to the new Volterra
  operator, giving
the upper
  bound
  \begin{equation} \label{critbound}
  |\partial_t u_{\lambda}^{\C}(t)| \leq \lambda \exp C \lambda |\Im t|
   \cdot \| u_{\lambda}\|_{L^{2}(\partial  \Omega)} .
   \end{equation}
    Only the exponential growth
   rate is significant here. We then   apply Proposition \ref{DFnew} to
   bound the number of zeros by the maximum of \begin{equation}
   \label{THIRD}  \log \frac{
   |\partial_t u_{\lambda}^{\C}(q^{\C}(t))|}{||\partial_t u_{\lambda} ||_{L^2(\partial
   \Omega}||}  \leq C \lambda + O(\log \lambda)
  + O(\log \frac{||u_{\lambda}||_{L^2(\partial \Omega)}}{||\partial_t u_{\lambda} ||_{L^2(\partial
   \Omega}||}). \end{equation}  Of course this estimate assumes $u_{\lambda}'(t) \not= 0$
  identically, as can happen with radial eigenfunctions on the
  disc. Assuming (\ref{NEUND}), the third term of (\ref{THIRD}) is $O(\lambda)$, completing
   the proof.

\end{document}